\documentclass[11pt]{amsart}
\usepackage{amssymb}

\theoremstyle{plain}
\newtheorem{theorem}{Theorem}[section]
\newtheorem{corollary}[theorem]{Corollary}
\newtheorem{lemma}[theorem]{Lemma}
\newtheorem{proposition}[theorem]{Proposition}
\newtheorem{definition}[theorem]{Definition}

\theoremstyle{remark}

\numberwithin{equation}{section}

\newcommand{\dbullet}{\,\raisebox{-1.4pt}{\( \bullet \)}\,}
\newcommand{\dowd}{\, .\,}
\newcommand{\res}{\upharpoonright}

\begin{document}
\title[Ramsey theory]{Abstract approach to finite Ramsey theory\\
and a self-dual Ramsey theorem}

\author{S{\l}awomir Solecki}

\thanks{Research supported by NSF grant DMS-1001623.}

\address{Department of Mathematics\\
University of Illinois\\
1409 W. Green St.\\
Urbana, IL 61801, USA}

\email{ssolecki@math.uiuc.edu}



\begin{abstract} We give an abstract approach to finite Ramsey theory and
prove a general Ramsey-type theorem. We deduce from it a self-dual
Ramsey theorem, which is a new result naturally generalizing both
the classical Ramsey theorem and the dual Ramsey theorem of Graham
and Rothschild. In fact, we recover the pure finite Ramsey theory
from our general Ramsey-type result in the sense that the classical
Ramsey theorem, the Hales--Jewett theorem (with Shelah's bounds),
the Graham--Rothschild theorem, the versions of these results for
partial rigid surjections due to Voigt, and the new self-dual Ramsey
theorem are all obtained as iterative applications of the general
result.
\end{abstract}

\maketitle

\section{Introduction}\label{S:intro}

\subsection{Abstract approach to Ramsey theory and its
applications}\label{Su:abap} We give an abstract approach to pure
(unstructured) finite Ramsey theory. The spirit of the undertaking
is similar to Todorcevic's approach to infinite Ramsey theory
\cite[Chapters 4 and 5]{To10}, even though, on the technical level,
the two approaches are different. There are three main points
to the present paper. The first one is the existence of a single,
relatively simple type of algebraic structure, called Ramsey domain
over a composition space (or a normed composition space), that underlies Ramsey
theorems. The second point is the
existence of a single Ramsey theorem, which is a result about the
algebraic structures just mentioned. Particular Ramsey theorems are
instances, or iterative instances, of this general result for
particular Ramsey domains, much like theorems about, say, modules
have particular instances for concrete modules. The latter point
opens up a possibility of classifying concrete Ramsey situations, at
least in limited contexts; see Section~\ref{S:pro}.
Finally, the third main point
is a new concrete Ramsey theorem obtained using the abstract approach to Ramsey theory.
We call this theorem the self-dual Ramsey theorem. It is a common generalization of
the classical Ramsey theorem and the dual Ramsey theorem.
As far as proofs of the known results are concerned, one advantage
of the approach given here is its uniformity.
Our approach also provides a hierarchy of the Ramsey
results according to the number of times the abstract Ramsey theorem
is applied in their proofs. For example, the classical Ramsey
theorem requires one such application, the Hales--Jewett theorem
requires two applications, the Graham--Rothschild theorem three, and the self-dual Ramsey theorem four.

The following vague observation is at the starting point of the abstract approach to Ramsey theory.
(This observation may not be apparent at this point, but it will become clear with the reading of the paper.)
Roughly speaking, a Ramsey-type theorem is a statement of the
following form.
We are given a set $S$ chosen arbitrarily from some
fixed family $\mathcal S$ and a number of colors $d$. We find a set
$F$ from another fixed family $\mathcal F$ with a ``scrambling"
function, usually a type of composition, defined on $F\times S$:
\[
F\times S \ni (f,x)\to f\dowd x\in F\dowd S.
\]
The arrangement is such that for each $d$-coloring of $F\dowd S$
there is $f\in F$ with $f\dowd S$ monochromatic. Our undertaking consists of finding a general,
algebraic framework in which we isolate an abstract pigeonhole
principle and prove that it implies a precise version of the above
abstract Ramsey-type statement.

The paper is structured as follows.

Section~\ref{S:intro}: Later in this introduction, we recall the central theorems of the finite unstructured Ramsey theory and place our new concrete
Ramsey result---the self-dual Ramsey theorem---in this context.

Section~\ref{S:injsur}: We fix the notions which are used to state the concrete Ramsey results in this paper.
These notions are formulated in the language of injections and surjections. We show how to translate
Ramsey statements as in Section~\ref{S:intro} to statements in this language.

Section~\ref{S:alg}: We define the new algebraic notions needed to
phrase and prove the abstract Ramsey theorem. The progression of notions is as follows:
\begin{enumerate}
\item[---] {\em actoid}, a most basic notion of action;

\item[---] {\em set actoid} over an actoid, a lift of the operations on an actoid to subsets;

\item[---] {\em composition spaces}, actoids with an added operator;

\item[---] {\em Ramsey domains}, set actoids over composition spaces fulfilling additional conditions.
\end{enumerate}

Section~\ref{S:RP}: Using the algebraic notions introduced in the previous section, we phrase the main Ramsey theoretic conditions. So the first
half of this section contains formulations of:
\begin{enumerate}
\item[---] the {\em Ramsey property} for Ramsey domains;

\item[---]the {\em pigeonhole principle} for Ramsey domains.
\end{enumerate}
In the second half of the section, we prove in Theorem~\ref{T:main2} and Corollary~\ref{C:mainco} that, under mild assumptions,
the pigeonhole principle implies the Ramsey property for Ramsey domains.

Section~\ref{S:loc}: We prove that, in many situations, to get a Ramsey theorem it
suffices to check only a localized, and therefore easier, version of
the abstract pigeonhole principle. We start this section by formulating:
\begin{enumerate}
\item[---] the {\em local pigeonhole principle} for Ramsey domains,
\end{enumerate}
and follow it by defining
\begin{enumerate}
\item[---] {\em normed composition spaces}, composition spaces with a norm to a partial ordering.
\end{enumerate}
After that we prove Theorem~\ref{T:hp} and Corollary~\ref{C:mainco2}
that show that the local pigeonhole principle, under mild conditions, implies the
Ramsey property for Ramsey domains over normed composition spaces.

Section~\ref{S:prop}: We show two results allowing us to propagate the pigeonhole
principle. In the first one, Proposition~\ref{P:prod}, we get the
pigeonhole principle for naturally defined products assuming it
holds for the factors. The second result, Proposition~\ref{P:inter},
involves a notion of interpretability and establishes preservation
of the pigeonhole principle under interpretability.

Section~\ref{S:coram}: We give examples of composition spaces and Ramsey domains. More
examples can be found in papers \cite{So3} and \cite{So4}. 
(Note that the terminology in \cite{So3} differs somewhat from the one in the present paper.)

Section~\ref{S:appl}: This section contains applications of the abstract Ramsey approach
to concrete situations.
As a consequence of the general theory we
obtain a new self-dual Ramsey theorem. We give its statement and
explain its relationship with other results in
Subsection~\ref{Su:self} below. Here, let us only mention one interesting
feature of the proof of this theorem: the role of the pigeonhole
principle is played by the Graham--Rothschild theorem. We also give other applications of the general
theory to concrete examples. We show how to derive as iterative
applications of the abstract Ramsey result the classical Ramsey
theorem, the Hales--Jewett theorem, the Graham--Rothschild theorem
as well as the versions of these results for partial rigid
surjections due to Voigt. We note that in the proof of the
Hales--Jewett theorem the bounds we obtain on the parameters
involved in it turn out to be primitive recursive and are
essentially the same as Shelah's bounds from \cite{Sh88}. We will,
however, leave it to the reader to check the details of this
estimate. More applications of the abstract approach to Ramsey theory
involving finite trees can be found in \cite{So3} and \cite{So4}.
(Note again that the terminology in the present paper differs from that in \cite{So3}.)

Section~\ref{S:wal}: In Theorem~\ref{T:color}, we give an interesting
example for which Ramsey theorem fails. The objects that
are being colored can be viewed as Lipschitz surjections with
Lipschitz constant $1$ between initial segments of the set of natural numbers.
This example is motivated by considerations in topological dynamics.

Section~\ref{S:pro}: We make concluding remarks and state a problem on
classifying Ramsey theorems in a natural, but limited, set-up.

It may be worthwhile to point out that, on the conceptual level, the elegant approach of
Graham, Leeb and Rothschild \cite{GrLeRo72} and of Spencer
\cite{Sp79} to finite Ramsey theorems for spaces is very much
different from the approach presented here. The differences on the
technical level are equally large. One main such difference is that,
unlike here, the setting of \cite{GrLeRo72} and \cite{Sp79} has a
concrete pigeonhole principle built into it, which in that approach
is the Hales--Jewett theorem.

The pure Ramsey theory, which is the subject matter of this paper,
is a foundation on which the Ramsey theory for finite structures is
built, but is not a part of it. Consequently, the methods of the
present paper have nothing directly to say about the structural
Ramsey theory as developed for relational structures by
Ne\v{s}et\v{r}il and R\"{o}dl in \cite{NeRo}, \cite{NeRo2},
\cite{NeRo3}, and by Pr\"{o}mel in \cite{Pr85} and, more recently,
for structures that incorporate both relations and functions by the
author in \cite{So}, \cite{So2}.

\subsection{Self-dual Ramsey theorem and finite Ramsey theory}\label{Su:self}

We consider $0$ to be a natural number. As is usual, we adopt the
convention that for a natural number $n$,
\[
[n] = \{ 1, \dots, n\}.
\]
In particular, $[0] = \emptyset$.

The aim of this subsection is to survey the fundamental results of finite Ramsey theory, in particular,
we recall the classical Ramsey theorem and the dual Ramsey theorem
of Graham and Rothschild. In their context, we state the new self-dual Ramsey theorem. Later we recall further Ramsey
theoretic results that will be relevant in the sequel.

The classical Ramsey theorem can be stated as follows.

\medskip
\noindent{\bf Ramsey's Theorem.} {\em Given the number of colors $d$ and natural numbers
$k$ and $l$, there exists a natural number $m$ such that for each
$d$-coloring of all subsets of $[m]$ of size $k$ there
exists a subset $B$ of $[m]$ of size $l$  such that
\[
\{ A\mid A\subseteq B \hbox{ and } A \hbox{ has size }k\}
\]
is monochromatic.}
\medskip

The dual Ramsey theorem, due to Graham and Rothschild \cite{GrRo71}, concerns partitions and can
be stated as follows. Recall that a partition $\mathcal P$ is {\em coarser than} a partition $\mathcal Q$
if each set in $\mathcal P$ is a union of some sets in $\mathcal Q$.

\medskip
\noindent{\bf Dual Ramsey Theorem.} {\em Given the number of colors $d$ and natural numbers $k, l >0$,
there exists a natural number $m$ such that for each
$d$-coloring of all partitions of $[m]$ with $k$ pieces
there exists a partition $\mathcal Q$ of $[m]$ with $l$ pieces such that
\[
\{ {\mathcal P}\mid {\mathcal P} \hbox{ a partition with $k$ pieces coarser than }{\mathcal Q}\}
\]
is monochromatic.}
\medskip

It is natural to ask if a ``self-dual" Ramsey theorem exists that
combines the two statements above. We formulate now such a self-dual
theorem. We will be coloring pairs consisting of a partition and a set
interacting with each other in a certain way.  Let $\mathcal R$ be a partition
of $[n]$ and let $C$ be a subset of $[n]$. Let $m\in {\mathbb N}$. We say that
$({\mathcal R}, C)$ is an $m$-{\em connection} if $\mathcal R$ and $C$ have
$m$ elements each and, upon listing ${\mathcal R}$ as $R_1, \dots, R_m$
with $\min R_i <\min R_{i+1}$ and $C$ as $c_1, \dots, c_m$ with
$c_i<c_{i+1}$, we have $c_i\in R_i$ for $i\leq m$ and $c_i<\min
R_{i+1}$ for $i<m$. We say that an $l$-connection $({\mathcal Q},
B)$ is an $l$-{\em subconnection} of an $m$-connection $({\mathcal
R}, C)$ if ${\mathcal Q}$ is a coarser partition than ${\mathcal R}$
and $B\subseteq C$.

Here is the self-dual Ramsey
theorem. Its reformulation in terms of surjections and injections is Theorem~\ref{T:RGR}. It is proved in Subsection~\ref{Su:selpr}.

\medskip
\noindent{\bf Self-dual Ramsey Theorem.} {\em Let $d>0$. For each $k,l\in {\mathbb N}$, $k,l>0$, there exists
$m\in {\mathbb N}$ such that for each $d$-coloring of all
$k$-subconnections of an $m$-connection $({\mathcal R}, C)$ there
exists an $l$-subconnection $({\mathcal Q}, B)$ of $({\mathcal R},
C)$ such that all $k$-subconnections of $({\mathcal Q}, B)$ get the
same color.}
\medskip

Ramsey's Theorem is just the theorem above for colorings
that do not depend on the first coordinate; Dual Ramsey Theorem
theorem is the above theorem for colorings that do not depend on the
second coordinate.

We recall now some other results of finite Ramsey theory, partly to remind the reader of the main results
of the theory, and partly because we will need them in our proof of the self-dual Ramsey theory. We will
illustrate the abstract approach to Ramsey theory developed in this paper by proving these results. This
exercise is to support our assertion that all the unstructured Ramsey theoretic results, including the self-dual
Ramsey theorem, can be obtained as iterative instances of the abstract Ramsey theorem.

We state these results here using the language of parameter sets as is done in Ne{\v s}et{\v r}il's
survey \cite{Ne95}. Later, in Section~\ref{S:injsur}, we will restate them in the language of injections and
surjections. Let $A, l, n\in {\mathbb N}$ with $A$, $l$ not both equal to $0$. By an {\em
$l$-dimensional $A$-parameter set on $n$} we understand a pair of the form
\begin{equation}\label{E:par}
V = (g, {\mathcal G}),
\end{equation}
where $\mathcal G$ consists of $l$ non-empty, pairwise disjoint
subsets of $[n]$ and $g\colon [n]\setminus \bigcup {\mathcal G} \to
[A]$.
Note that if $V$ is of dimension $0$, then ${\mathcal G}=\emptyset$, and
$V = (g, \emptyset)$ can, and will, be identified with the function $g\colon [n]\to A$.
The set of all such functions is denoted by $A^n$. Note also that if $A=0$, then there is
only one choice for $g$---the empty function---and $V$ can be identified with $\mathcal G$, which
in this case is a partition of $[n]$ into $l$ pieces. Thus,
$0$-parameter sets are identified with partitions.

A $k$-dimensional $A$-parameter set on $n$
\[
U = (f, {\mathcal F})
\]
is an {\em $A$-parameter subset of $V$} as in \eqref{E:par} if each set in
${\mathcal F}$ is the union of some sets in $\mathcal G$, $f$
extends $g$, and $f$ restricted to each set in $\mathcal G$ included
in its domain is constant. In the particular case, when $U$ is $0$-dimensional and
is identified with $f\colon [n]\to A$, we say that $f$ is {\em a function in} $V$.
This translates to $f$ being an extension of $g$ to $[n]$ that is constant on each set in $\mathcal G$.
In the particular case, when $V$ is a $0$-parameter set and, therefore, so is $U$, and they are both
identified with the partitions $\mathcal G$ and $\mathcal F$, respectively, we have that $\mathcal F$ is
coarser than $\mathcal G$.

The first theorem is from \cite{HeJe63}.

\medskip
\noindent{\bf Hales--Jewett Theorem.} {\em Fix $A\in {\mathbb N}$, $A>0$, and $d>0$. For each $l \in
{\mathbb N}$ there exists $m\in {\mathbb N}$ such that for each
$d$-coloring of $A^m$ there exists an $l$-dimensional $A$-parameter
set $U$ such that all functions in $U$ get the same color.}
\medskip

The second theorem is from \cite{GrRo71}.

\medskip
\noindent{\bf Graham--Rothschild Theorem.} {\em Fix $A\in {\mathbb N}$ and $d>0$. Given $k, l\in
{\mathbb N}$, with $k,l>0$ if $A=0$, there exists $m\in {\mathbb N}$ such that for each
$d$-coloring of all $k$-dimensional $A$-parameter subsets of an $m$-dimensional
$A$-parameter set $V$ there exists an $l$-dimensional $A$-parameter
subset $U$ of $V$ such that all $k$-dimensional $A$-parameter
subsets of $U$ get the same color.}
\medskip

Note that the Dual Ramsey Theorem is an instance of the Graham--Rothschild Theorem for $A=0$ and $k,l>0$.

We will also need the ``partial" versions of the above results. These partial versions were
established by Voigt in \cite{Vo80}. In order to state them, we need to generalize the definitions above.
A $k$-dimensional $A$-parameter set on $m$
\[
U = (f, {\mathcal F})
\]
is a {\em partial $A$-parameter subset of $V$} as in \eqref{E:par} if $m=n$ or
$m+1 =\min a$ for some $a\in {\mathcal G}$, and $U$ is an $A$-parameter subset of
the $A$-parameter set
\[
\left( g\res ([m]\setminus \bigcup {\mathcal G}),\, \{ b\cap [m]\mid b\in {\mathcal G}, \, b\cap [m]\not=\emptyset\} \right).
\]
Note that each $A$-parameter subset of $V$ is a partial $A$-parameter subset of $V$. When $U$ is $0$-dimensional
and is identified with the function $f\colon [m]\to A$,  we say that $f$ is {\em a partial function in} $V$ as in \eqref{E:par}. The notion
of partial function in $V$ can be rephrased by saying that $m=n$ or
$m+1=\min a$ for some $a\in {\mathcal G}$, $f$ extends $g\res ([m]\setminus \bigcup {\mathcal G})$ to $[m]$,
and $f$ is constant on each $b\cap [m]$ for $b\in {\mathcal G}$.

Let $A^{\leq m}$ stand for the set of all
functions $f\colon [m']\to A$ with $m'\leq m$.

\medskip
\noindent{\bf Hales--Jewett Theorem, Voigt's version.} {\em Fix $A\in {\mathbb N}$ and $d>0$. For each $l \in
{\mathbb N}$ there exists $m\in {\mathbb N}$ such that for each
$d$-coloring of $A^{\leq m}$ there exists an $l$-dimensional $A$-parameter
set $U$ on some $m'\leq m$ such that all partial functions in $U$ get the same color.}

\medskip
\noindent{\bf Graham--Rothschild Theorem, Voigt's version.} {\em Fix $A\in {\mathbb N}$ and $d>0$. For each $k, l\in
{\mathbb N}$ there exists $m\in {\mathbb N}$ such that for each
$d$-coloring of all $k$-dimensional partial $A$-parameter subsets of an $m$-dimensional
$A$-parameter set $V$ there exists an $l$-dimensional partial $A$-parameter
subset $U$ of $V$ such that all $k$-dimensional partial $A$-parameter
subsets of $U$ get the same color.}

\subsection{Walks} Uspenskij in \cite{Us00} asked for an identification of the universal minimal flow of
the homeomorphism group of the generic compact connected metric space called the pseudo-arc. In view
of the papers by Irwin and the author \cite{IrSo05} and by Kechris, Pestov, and Todorcevic \cite{KePeTo05},
it became apparent to the author that such an identification can be accomplished if a certain Ramsey statement were true. However, in Section~\ref{S:wal}, we
prove that this Ramsey statement is false. We state this theorem below after we have defined the objects involved in it.
The theorem gives an interesting and natural class of rigid surjections (for a definition see Section~\ref{S:injsur}) for
which Ramsey theorem fails. The coloring that make it fail was found as a by-product of an analysis of the Ramsey statement
with the abstract approach.

A {\em walk} is a function $s\colon [n]\to [m]$ that is surjective and such that $s(1)=1$ and $|s(i+1)-s(i)|\leq 1$ for
all $i\in [n-1]$. The following theorem is proved as Theorem~\ref{T:color}.

\medskip
\noindent{\bf Coloring of Walks Theorem.}
{\em Let $m\geq 3$. There exists a coloring with two colors of all walks from
$[m]$ to $[3]$  such that for each walk $t\colon [m]\to [6]$ the set
\[
\{ s\circ t\mid s\colon [6]\to [3]\hbox{ a walk}\}
\]
is not monochromatic.}

\section{The language of injections and surjections; reformulations of Ramsey results}\label{S:injsur}

In the paper, we consistently use the language
of rigid surjections and increasing injections rather than that of
partitions and sets. (This language was proposed in \cite{PV1}.) In
our opinion, this choice is more satisfying from the theoretical
point of view and it easily accommodates
objects coming from topology such as walks in
Section~\ref{S:wal}. Note, however, that the abstract approach
is also applicable to the partition-and-set formalism. A
canonical way of translating statements in one language into the
other is explained in Subsection~\ref{Su:transl}.

Recall that for $N\in {\mathbb N}$, we let
\[
[N] = \{ 1, \dots, N\}.
\]
In the sequel, we use letters $K,\,L,\,M,\,N,\,P,\,Q$, possibly with
subscripts, to stand for natural numbers.

\subsection{Classes of injections and surjections}\label{Su:injsur}

We fix here some notation and some notions needed in the sequel.

By an {\em increasing injection} we understand a {\em strictly}
increasing function from $[K]$ to $\mathbb N$ for some $K\in
{\mathbb N}$. Let
\begin{equation}\label{E:inin}
{\rm II}= \{ i\mid i\hbox{ is an increasing injection}\}.
\end{equation}
For $K\leq L$, let
\[
\binom{L}{K} = \{ i\in {\rm II}\mid i\colon [K] \to [L]\}.
\]
Since an increasing injection from $[K]$ to $[L]$ is determined by,
and of course itself determines, its image, that is, a $K$ element subset
of $[L]$, the set $\binom{L}{K}$ defined above can be thought of as
the set of all $K$ element subsets of $[L]$.

Let
\[
{\rm S} = \{ v\mid \exists K,L\; (K\leq L\hbox{ and } v\colon
[L]\to [K] \hbox{ is a surjection})\}.
\]
We adopt the convention that for $v\in S$ writing $v\colon [L]\to
[K]$ signifies that $v$ is onto $[K]$.

A {\em rigid surjection} is a
function $s\colon [L]\to [K]$ that is surjective and such that for
each $y\in [L]$ there is $x\in [K]$ with $s([y]) = [x]$.
In other words, for each $y\in [L]$, we have
\[
s(y)\leq 1+\max_{1\leq x<y} s(x)
\]
with the convention that the maximum over the empty set is $0$. Note
that the notion of rigid surjection is obtained simply by dualizing
the notion of increasing injection: increasing injections are {\em
injections} such that {\em preimages} of initial segments are
initial segments, while rigid surjections are {\em surjections} such
that {\em images} of initial segments are initial segments.
Let
\begin{equation}\notag
{\rm RS} = \{ s\in {\rm S} \mid s \hbox{ is a rigid surjection}\}.
\end{equation}

By a {\em walk} we understand a function $s\colon [L]\to [K]$ such that $s(1)=1$ and
for each $x\in [L-1]$,
\[
s(x)-1\leq s(x+1)\leq s(x)+1.
\]
Clearly each walk is a rigid surjection. Set
\[
W = \{ s\mid s\hbox{ is a walk}\}.
\]

An {\em increasing surjection} is a function $p\colon [L]\to [K]$ that
is surjective and such that for $y_1, y_2\in [L]$ if $y_1\leq y_2$,
then $p(y_1)\leq p(y_2)$, so strictly speaking $p$ is a
non-decreasing surjection. Clearly each increasing surjection is a walk. Let
\begin{equation}\label{E:IS}
{\rm IS} = \{ p \in {\rm S}\mid p\hbox{ is an increasing
surjection}\}.
\end{equation}

We also have a notion that combines surjections and injections.
Let $K, L$ be natural numbers. We call
a pair $(s,i)$ a {\em connection between $L$ and $K$} if $s\colon
[L]\to [K]$, $i\colon [K]\to [L]$ and for each $x\in [K]$
\[
s(i(x)) = x\;\hbox{ and }\; \forall y<i(x)\; s(y)\leq x.
\]
In other words, $i$ is a left inverse of $s$ with the additional
property that at each $x\in [K]$ the value $i(x)$ is picked only
from among those elements of $s^{-1}(x)$ that are ``visible from
$x$," that is, from those $y'\in s^{-1}(x)$ for which
\[
s\upharpoonright (\{ y\mid y<y'\}) \leq x.
\]
We write
\[
(s,i)\colon [L]\leftrightarrow [K].
\]
It is easy to see that if $(s, i)$ is a connection, then $i$ is an
increasing injection and $s$ is a rigid surjection. Also for each
rigid surjection $s$ there is an increasing injection $i$ (usually
many such injections) for which $(s,i)$ is a connection, and for each
increasing injection $i$ there is a rigid surjection $s$ (again,
usually many such surjections) with $(s,i)$ a connection.

Finally, for technical reasons, we need the notion of augmented surjection, which are
ordered pairs whose elements are a rigid surjection and an
increasing surjection with these elements appropriately interacting
with each other. Let
\begin{equation}\notag
\begin{split}
{\rm AS} = \{ (s,p)\mid \exists K,L\in {\mathbb N}\;(&s, p\colon
[L]\to [K],\, p\in {\rm IS},\, s\leq p,\\
&\forall x\in [K]\; s(\max p^{-1}(x)) = x)\}.
\end{split}
\end{equation}
It is easy to see that $(s,p)\in {\rm AS}$ implies that $s$ is a
rigid surjection. Elements of ${\rm AS}$ are called {\em augmented
surjections}.

\subsection{The canonical composition of a rigid surjection and a function}\label{Su:canco}

In the sequel in various situations, we will require to compose functions with rigid surjections. It will always be
done in a particular way. This canonical composition will be a restriction of the usual composition to a certain initial segment. Here
is a precise definition.
Let $v\colon [L]\to [K]$ be a function and let $s\colon [N]\to
[M]$ be a rigid surjection. The {\em canonical composition} of $v$
and $s$, which we denote by $v\circ s$,
is defined if and only if $L\leq M$. In this case, let $N_0\leq N$ be the largest number such
that $s(y)\leq L$ for all $y\leq N_0$. Define
\begin{equation}\label{E:ccom}
v\circ s
\end{equation}
to be the usual composition of $v$ with $s\upharpoonright [N_0]$.
It is easy to see that $v\circ s$ is the restriction of the usual
composition of $v$ and $s$, which is a partial function on $[N]$, to the largest initial segment of $[N]$
on which this composition is defined. If $M=K$, then $v\circ s$ is the usual composition of $v$ and $s$.

Note that if $v$ is a surjection, then $v\circ s\colon
[N_0]\to [K]$ is a surjection. If $v$ is a rigid surjection, then
$v\circ s$ is a rigid surjection.
If $v$ and $s$ are walks, then so is $v\circ s$.
If $v$ and $s$ are
increasing surjections, then so is $v\circ s$.

It is easy to verify that if $v$ is a function, $s$, $t$ are rigid
surjections, and $(v\circ s)\circ t$ and $v\circ (s\circ t)$ are
both defined, then
\[
(v\circ s)\circ t= v\circ (s\circ t).
\]
This observation will be frequently used in the sequel.

\subsection{The self-dual Ramsey theorem and other Ramsey-type theorems in terms of injections and surjections}\label{Su:refor}

We present here reformulations in the language of surjections and injections of the Ramsey theorems, including
the self-dual Ramsey theorem, presented in Subsection~\ref{Su:self}. We explain how the translation works in Subsection~\ref{Su:transl}. .

Given connections $(s,i)\colon [L]\leftrightarrow [K]$ and
$(t,j)\colon [M]\leftrightarrow [L]$, define
\[
(t,j)\cdot (s,i) \colon [M]\leftrightarrow [K]
\]
as
\[
(s\circ t, j\circ i).
\]
Note that the orders of the compositions in the two coordinates are
different from each other. One sees easily that the composition of
two connections is a connection.

The following theorem is a reformulation of the self-dual Ramsey theorem from Subsection~\ref{Su:self}.
We think of it as the official statement of the theorem. Its proof is given in Subsection~\ref{Su:selpr}.

\begin{theorem}\label{T:RGR}
Let $d>0$ be a natural number. Let $K$ and $L$ be natural numbers.
There exists a natural number $M$ such that for each $d$-coloring of
all connections between $M$ and $K$ there is $(t_0,j_0)\colon
[M]\leftrightarrow [L]$ such that
\[
\{ (t_0,j_0)\cdot (s,i) \mid (s,i)\colon [L]\leftrightarrow [K]\}
\]
is monochromatic.
\end{theorem}

Below we enclose a list of Ramsey-type theorems stated in Subsection~\ref{Su:self} formulated here
in the language of increasing injections and rigid surjections. These statements of the theorems
will be proved and used in this paper.

\medskip

\noindent {\bf Ramsey's Theorem.} {\em Given $d>0$ and natural numbers
$K$ and $L$, there exists a natural number $M$ such that for each
$d$-coloring of all increasing injections from $[K]$ to $[M]$ there
exists an increasing injection $j_0\colon [L]\to [M]$ such that
\[
\{ j_0\circ i\mid i\colon [K]\to [L]\hbox{ an increasing
injection}\}
\]
is monochromatic.}

\medskip

\noindent {\bf Dual Ramsey Theorem.} {\em Given $d>0$ and natural numbers $K$
and $L$, there exists a natural number $M$ such that for each
$d$-coloring of all rigid surjections from $[M]$ to $[K]$ there
exists a rigid surjection $t_0\colon [M]\to [L]$ such that
\[
\{ s\circ t_0\mid s\colon [L]\to [K]\hbox{ a rigid surjection}\}
\]
is monochromatic.}

\medskip

\noindent {\bf Hales--Jewett Theorem.} {\em Given $d>0$ and $0< A\leq L$,
there exists $M\geq A$ with the following property. For each
$d$-coloring of the set
\[
\{ v\colon [M]\to [A]\mid v\upharpoonright [A] = {\rm id}_{[A]}\}
\]
there exists a rigid surjection $s_0\colon [M]\to [L]$ such that
$s_0\upharpoonright [A] ={\rm id}_{[A]}$ and
\[
\{ v\circ s_0\mid v\colon [L]\to [A],\, v\upharpoonright [A] = {\rm id}_{[A]}\}
\]
is monochromatic.}

\medskip

\noindent {\bf Hales--Jewett Theorem, Voigt's version.} {\em Given $d>0$,
$0< A\leq L$,
there exists $M\geq A$ with the following property. For each
$d$-coloring of the set
\[
\{ v\colon [M']\to [A]\mid A\leq M'\leq M\hbox{ and }v\upharpoonright [A] = {\rm id}_{[A]}\}
\]
there exists a rigid surjection $s_0\colon [M']\to [L]$ for some $M'\leq M$ such that
$s_0\upharpoonright [A] ={\rm id}_{[A]}$ and
\[
\{ v\circ s_0\mid v\colon [L']\to [A],\, A\leq L'\leq L,\hbox{ and }v\upharpoonright [A] = {\rm id}_{[A]}\}
\]
is monochromatic.}

\medskip

\noindent {\bf Graham--Rothschild Theorem.} {\em Given $d>0$, $A\leq K$,and $A \leq L$, there
exists $M\geq A$ with the following property. For each
$d$-coloring of
\[
\{ s\colon [M]\to [K]\mid s\in {\rm RS}\hbox{ and }
s\upharpoonright [A] = {\rm id}_{[A]}\}
\]
there is a rigid surjection $t_0\colon [M]\to [L]$, with
$t_0\upharpoonright [A] ={\rm id}_{[A]}$, such that
\[
\{ s\circ t_0\mid s\colon [L]\to [K],\, s\in {\rm RS}\hbox{ and }
s\upharpoonright [A] = {\rm id}_{[A]}\}
\]
is monochromatic.}

\medskip
\noindent {\bf Graham--Rothschild Theorem, Voigt's version.} {\em Given
$d>0$, $A\leq K$, and $A\leq L$, there exists $M\geq A$ with the following
property. For each $d$-coloring of
\[
\{ s\colon [M']\to [K]\mid A\leq M'\leq M, \, s\in {\rm RS}, \hbox{ and }
s\upharpoonright [A] = {\rm id}_{[A]}\}
\]
there exist $M_0'$ and a rigid surjection $t_0\colon [M_0']\to [L]$,
with $A\leq M_0'\leq M$ and $t_0\upharpoonright [A] ={\rm
id}_{[A]}$, such that
\[
\{ s\circ t_0\mid s\colon [L']\to [K],\, A\leq L'\leq L ,\, s\in
{\rm RS}, \hbox{ and } s\upharpoonright [A] = {\rm id}_{[A]}\}
\]
is monochromatic.}

\subsection{Translation of rigid surjections into parameter sets}\label{Su:transl}

We show here how to translate statements involving parameter sets
(sometimes called combinatorial cubes) into statements about rigid
surjections. This latter language was proposed by Pr\"{o}mel and
Voigt \cite{PV1}. It has been used in papers \cite{So} and
\cite{So2} in the context of structural Ramsey theory.

With each $A$-parameter set of dimension $l$ on $n$ as in
\eqref{E:par}, we associate a rigid surjection $s_V\colon [A+n]\to
[A+l]$ as follows. We enumerate the sets in $\mathcal G$ as $Y_1,
\dots, Y_l$ so that $\min Y_i<\min Y_j$ if $i<j$. We let
\[
s_V\upharpoonright [A] = {\rm id}_{[A]}
\]
and
\begin{equation}\notag
s_V(A+x) =
\begin{cases}
A+i, &\text{if $x\in Y_i$;}\\
g(x), &\text{if $x\in [n]\setminus\bigcup_{i=1}^l Y_i$.}
\end{cases}
\end{equation}
This association is a bijection between all $A$-parameter sets of
dimension $l$ on $n$ and all rigid surjections $s\colon [A+n]\to
[A+l]$ with the property $s\upharpoonright [A] = {\rm id}_{[A]}$.
Moreover, it is not difficult to check, and we leave it to the
reader, that an $A$-parameter set $U$ of dimension $k$ on $n$ is a
subobject of an $A$-parameter set $V$ of dimension $l$ on $n$ if and
only if there is a rigid surjection $r\colon [A+l]\to [A+k]$ with
$r\upharpoonright [A] = {\rm id}_{[A]}$ and such that
\begin{equation}\label{E:compr}
s_U = r\circ s_V,
\end{equation}
and for each rigid surjection $r\colon [A+l]\to [A+k]$ with
$r\upharpoonright [A] = {\rm id}_{[A]}$ there is an $A$-parameter
set $U$ of dimension $k$ on $n$ that is a subobject of $V$ such that
\eqref{E:compr} holds.

These remarks give translations between the statements of the
Hales--Jewett theorem, the Graham--Rothschild theorem, their Voigt's versions,
and the self-dual Ramsey theorem phrased in terms of
parameter sets as in Subsection~\ref{Su:self} and the statements of these results phrased
in terms of rigid surjections as in Subsection~\ref{Su:refor}.

\section{Algebraic structures}\label{S:alg}

We introduce here the main algebraic notions needed in the abstract approach to Ramsey theory. We illustrate the new
notions with two series of examples, one related to the classical Ramsey
theorem, the other one to the Hales--Jewett theorem.

\subsection{Actoids}

The notion of actoid defined below is the most rudimentary
version of an action much like a semigroup action on a set.

\begin{definition}
By an {\em actoid} we understand two sets $A$ and $Z$, a
partial binary function from $A\times A$ to $A$:
\[
(a, b)\to a\cdot b,
\]
and a partial binary function from $A\times Z$ to $Z$:
\[
(a,z)\to a\dowd z
\]
such that for $a,b\in A$ and $z\in Z$ if $a\dowd (b\dowd z)$ and
$(a\cdot b)\dowd z$ are both defined, then
\begin{equation}\label{E:acct}
a\dowd (b\dowd z) = (a\cdot b)\dowd z.
\end{equation}
\end{definition}
The binary operation $\cdot$ on an actoid as above will be
called {\em multiplication} and the binary operation $\dowd$ will be
called {\em action}. Unless otherwise stated, the multiplication
will be denoted by $a\cdot b$ and the action by $a\dowd z$.
Note that in the case when $A=Z$ and the multiplication
coincides with the action, an actoid becomes what sometimes is called a partial semigroup.
With some abuse of notation, we denote an actoid as in the
definition above by $(A,Z)$.

To gain some intuitions about actoids, one may think of both
$A$ and $Z$ as sets of functions with multiplication $a\cdot b$ on
$A$ corresponding to composition $a\circ b$ that is defined only
when the range of $b$ is included in the domain of $a$. Similarly,
the action of $A$ on $Z$, $a\dowd z$, corresponds to composition
$a\circ z$ that is defined when the range of $z$ is included in the
domain of $a$.

We will have two sequences of examples illustrating the main notions
that are being introduced: sequence A leads to the classical Ramsey
theorem, sequence B leads to the Hales--Jewett theorem.

\medskip
\noindent {\bf Example A1.} Recall the definition of increasing injections and their class $\rm II$ from Subsection~\ref{Su:injsur}.
We let $B=Y={\rm II}$ and we make $(B,Y)$ into a composition space as
follows. For $i, j\in {\rm II}$, $j\cdot i$ and $j\dowd i$ are
defined if ${\rm range}(i)\subseteq {\rm domain}(j)$ and then
\[
j\cdot i = j\dowd i =  j\circ i.
\]

\medskip

\noindent {\bf Example B1.}
Fix $K_0\in {\mathbb N}$. Recall the set of increasing surjections ${\rm IS}$ from Subsection~\ref{Su:injsur}. Let
\[
X_{K_0} = \{ f\mid \exists L\in {\mathbb N}\; (f\colon [L]\to \{
0\}\cup [K_0])\}
\]
For $p\colon [N]\to [M]$, $p\in {\rm IS}$ and $f\colon [L]\to \{
0\}\cup [K_0]$, $f\in X_{K_0}$, declare $p\dowd f$ to be defined
precisely when the canonical composition $f\circ p$ is defined, and let
\[
p\dowd f = f\circ p.
\]
It is easy to see that $({\rm IS}, X_{K_0})$ is an actoid.

\subsection{Set actoids}

We will need to lift the operations on a given actoid to families of sets.
This is done by restricting the obvious pointwise liftings and is the content of
the following definition of {\em set actoid}. There is no
harm, as far as applications in this paper go, in thinking about $\mathcal F$ and
$\mathcal S$ in the definition below as consisting of finite
non-empty sets.

\begin{definition}\label{D:fulac}
Let $(A, Z)$ be an actoid. Let $\mathcal F$ be a family of
subsets of $A$ and $\mathcal S$ a family of subsets of $Z$. Let
\[
(F,G)\to F\bullet G
\]
be a partial function from ${\mathcal F}\times {\mathcal F}$ to
$\mathcal F$ and let
\[
(F,S)\to F\dbullet S
\]
be a partial function from ${\mathcal F}\times {\mathcal S}$ to
$\mathcal S$. We say that $(\mathcal F, \mathcal S)$ with these two
operations is an {\em set actoid over} $(A, Z)$ provided that
whenever $F\bullet G$ is defined, then $f\cdot g$ is defined for all $f\in F$ and $g\in G$ and
\[
F\bullet G = \{ f\cdot g \mid f\in F,\, g\in G\},
\]
and whenever $F\dbullet S$ is defined, then $f\dowd x$ is defined for all $f\in F$ and $x\in S$ and
\[
F\dbullet S = \{ f\dowd x\mid f\in F,\, x\in S\}.
\]
\end{definition}
Because of \eqref{E:acct}, for $F, G, S$ such that both
$F\dbullet (G\dbullet S)$ and $(F\bullet G)\dbullet S$ are defined, one has
$F\dbullet (G\dbullet S)= (F\bullet G)\dbullet S$.

We adopt the following conventions for an actoid $(A,Z)$. For sets $F,G\subseteq A$,
we say that $F\cdot G$ is defined if $f\cdot g$ is defined for all $f\in F$ and $g\in G$, and we let
\begin{equation}\label{E:setmul}
F\cdot G= \{ f\cdot g\mid f\in F,\, g\in G\}.
\end{equation}
Similarly, for $F\subseteq A$ and $S\subseteq Z$, we say that $F\dowd S$ is defined, if $f\dowd x$ is defined for all $f\in F$ and $x\in S$, and we let
\begin{equation}\label{E:seteval}
F\dowd S = \{f\dowd x\mid f\in F, x\in S\}.
\end{equation}
If $F = \{ f\}$ for some $f\in A$, we write
\[
f\dowd S
\]
for $\{ f\}\dowd S$, if it is defined.
A set actoid as in Definition~\ref{D:fulac} is a restriction of this natural pointwise operations on sets.

We record the following easy lemma.
\begin{lemma}\label{L:asset}
Let $(A, Z)$ be an actoid. For $F,G\subseteq A$ and $S\subseteq
Z$ if $(F\cdot G)\dowd S$ and $F\dowd (G\dowd S)$ are both defined,
then they are equal and, moreover, for $f\in F, g\in G, x\in S$
\[
(f\cdot g)\dowd x = f\dowd (g\dowd x).
\]
\end{lemma}
The lemma above says, in particular, that the pair consisting of the
family of all subsets of $A$ and the family of all subsets of $Z$
with the operations defined by \eqref{E:setmul} and
\eqref{E:seteval} is an actoid in its own right.

\medskip

\noindent {\bf Example A2.}
We continue with Example~A1. Recall the definition of $\binom{L}{K}$ from Subsection~\ref{Su:injsur}.
We let ${\mathcal F} = {\mathcal S}$ consist of all subsets of $\rm
II$ of the form $\binom{L}{K}$ with $1\leq K\leq L$ or $K=L=0$.
Note that $\binom{0}{0}$ contains only one element---the empty function.
For $\binom{L}{K}, \binom{N}{M}\in {\mathcal F} = {\mathcal S}$, we
let $\binom{N}{M}\bullet\binom{L}{K}$ and $\binom{N}{M}\dbullet
\binom{L}{K}$ be defined if and only if $L=M$, and then we let
\[
\binom{N}{L}\bullet \binom{L}{K} = \binom{N}{L}\dbullet \binom{L}{K}
= \binom{N}{K}.
\]
One easily checks that $({\mathcal F}, {\mathcal S})$ with $\bullet$
and $\dbullet$ is a set actoid over $(B,Y)$. Note that
\[
\binom{N}{M}\cdot \binom{L}{K} = \binom{N}{M}\dowd \binom{L}{K}
\]
are defined under a weaker assumption that $M\geq L$ and in that
case they are both equal to $\binom{N-(M-L)}{K}$.

\medskip

\noindent {\bf Example B2.}
We continue with Example~B1;
in particular, $K_0\in {\mathbb N}$ remains fixed. We assume
$K_0\geq 1$. For $K\geq 2$, let $S_K\subseteq X_{K_0}$ consist of
all $h\colon [K]\to \{ 0\}\cup [K_0]$ such that for some $1\leq
a\leq b<K$ and $0\leq c\leq K_0$, we have
\begin{equation}\label{E:neun}
h(x) =
\begin{cases}
K_0,&\text{if $x\leq a$;}\\
c, &\text{if $a+1\leq x\leq b$;}\\
\max(1, K_0- 1), &\text{if $b+1\leq x$.}
\end{cases}
\end{equation}
Formula \eqref{E:neun} always gives $h(1)=K_0$ and $h(K)=\max(1, K_0- 1)$.
Let $S_K'\subseteq X_{K_0}$ consist of all $h\colon [K]\to \{0\}\cup [K_0]$ such that for some $1\leq
a\leq b<K$ and $0\leq c\leq \max(1, K_0- 1)$
\begin{equation}\label{E:neun2}
h(x) =
\begin{cases}
c, &\text{if $a+1\leq x\leq b$;}\\
\max(1, K_0- 1), &\text{if $x\leq a$ or $b+1\leq x$}.
\end{cases}
\end{equation}
Formula \eqref{E:neun2} implies that the image of $h$ is included in $\{ 0\}\cup [\max (1, K_0-1)]$.
Let
\[
{\mathcal S}_{K_0}= \{ S_K\mid K\geq 2\} \cup \{ S_K'\mid K\geq 2\}.
\]

For $0< K\leq L$, let
\[
F_{L,K} = \{ p\in {\rm IS}\mid p\colon [L]\to [K]\}.
\]
Let
\[
{\mathcal F}_0 = \{ F_{L,K}\mid 0< K\leq L\}.
\]

We let $F_{N,M}\bullet F_{L,K}$ be defined if $L=M$ and then we let
\[
F_{N,L}\bullet F_{L,K} = F_{N,K}.
\]
We let $F_{M,L}\dbullet S_K$ and $F_{M,L}\dbullet S_K'$ be defined precisely when $K=L$ and, in
this case, we let
\[
F_{M,K}\dbullet S_K = S_M\;\hbox{ and }\; F_{M,K}\dbullet S_K'= S_M'.
\]
It is now easy to check that $({\mathcal F}_0, {\mathcal S}_{K_0})$
with the operations defined above is a set actoid over
$({\rm IS}, X_{K_0})$.

\subsection{Composition spaces}

To formulate the pigeonhole principle, we need additional structure
on actoids.

\begin{definition}
A {\em composition space} is an actoid $(A, Z)$ together with a
function $\partial\colon Z\to Z$ such that for $a\in A$ and $z\in
Z$, if $a\dowd z$ is defined, then $a\dowd \partial z$ is defined
and
\begin{equation}\label{E:acrestr}
a\dowd \partial z =\partial (a\dowd z).
\end{equation}
\end{definition}

This additional function on $Z$ will be called {\em truncation} and
it will always be denoted by $\partial$ possibly with various
subscripts and superscripts.

Condition~\eqref{E:acrestr} in the above definition states that in a
composition space $(A,Z)$ the action of $A$ on $Z$ is done by partial
homomorphisms of the structure $(Z, \partial)$. If we continue to
think of an actoid $(A,Z)$ as a family of functions $A$ acting
by composition on a family of functions $Z$, then we can view
truncation as a ``restriction operator" on functions from $Z$. So,
condition~\eqref{E:acrestr} can be translated to say that if the
composition of $a$ and $z$ is defined, then so is the composition of
$a$ and the restriction $\partial z$ of $z$ and its result is a
restriction of the composition of $a$ and $z$, which we require to
be given by the operator $\partial$. Truncation can also be thought
as producing out of an object $z$ a simpler object $\partial z$ of
the same kind. In proofs, this point of view leads to inductive
arguments.

We write
\[
\partial^t z
\]
for the element obtained from $z$ after $t\in {\mathbb N}$
applications of $\partial$. For a subsets $S\subseteq Z$, we write
\begin{equation}\label{E:setprim}
\partial S = \{ \partial z \mid z\in S\}.
\end{equation}
Again, for $t\in {\mathbb N}$, we write
\[
\partial^t S
\]
for the result of applying the operation $\partial$ to $S$ $t$
times.

We record the following obvious lemma.

\begin{lemma}\label{L:aba}
Let $(A, Z)$ be a composition space. Then for $F\subseteq A$ and
$S\subseteq Z$, if $F\dowd S$ is defined, then $F\dowd \partial S$
is defined and
\[
\partial (F\dowd S) = F\dowd (\partial S).
\]
\end{lemma}

It follows from the above lemma that if $(A,Z)$ is a composition space,
then the pair consisting of the family of all subsets of $A$ and the
family of all subsets of $Z$ becomes a composition space with the
operations defined by \eqref{E:setmul}, \eqref{E:seteval}, and
\eqref{E:setprim}.

\medskip

\noindent {\bf Example A3.}
We continue with Examples~A1 and A2.
For $i\colon [K]\to {\mathbb N}$ in $Y={\rm II}$, let
\[
\partial i = i\upharpoonright [\max(0, K- 1)].
\]
It is easy to check that $(B,Y)$ with $\partial$ is a composition space.

\medskip

\noindent {\bf Example B3.}
We continue with Examples~B1 and B2.
Fix $K_0\geq 1$. On $X_{K_0}$, we define the following truncation. For
$f\in X_{K_0}$, $f\colon [L]\to \{ 0\}\cup [K_0]$, let
\begin{equation}\notag
(\partial f)(x) =
\begin{cases}
\max(1, K_0- 1),&\text{ if $f(x)=K_0$;}\\
f(x),&\text{ if $f(x)\leq K_0-1$.}
\end{cases}
\end{equation}
It is now easy to see that for $p\in {\rm IS}$, if $p\dowd f$ is
defined, then
\[
\partial(p\dowd f) = \partial(f\circ p) = (\partial f)\circ p =
p\dowd (\partial f),
\]
and, therefore, $({\rm IS}, X_{K_0})$ with $\partial$ defined above
is a composition space.

\subsection{Ramsey domains}

The additional structure on composition spaces will allow us to formulate conditions on
set actoids that are needed for our Ramsey result. A set actoid fulfilling these
conditions will be called a Ramsey domain. But, first we need to introduce the following notion.
\begin{definition}
Let $(A, Z)$ be an actoid. For $a,b\in A$, we say that {\em $b$
extends $a$} if for each $x\in Z$ for which $a\dowd x$ is defined,
we have that $b\dowd x$ is defined and $a\dowd x = b\dowd x$.
\end{definition}

Now we introduce the main notion of this subsection.

\begin{definition}\label{D:ramdom} A set actoid $({\mathcal F},\, {\mathcal S})$ over a composition space is called a {\em Ramsey domain}
if each set in $\mathcal S$ is non-empty and the following conditions hold for all $F,G\in {\mathcal F}$ and $S\in {\mathcal S}$,
\begin{enumerate}
\item[(i)] if $F\dbullet (G\dbullet S)$ is defined, then so is $(F\bullet G)\dbullet S$;

\item[(ii)] $\partial S\in {\mathcal S}$;

\item[(iii)] if $F\dbullet\partial S$ is defined, then there is $H\in {\mathcal F}$ such that $H\dbullet S$ is defined and for each $f\in F$
there is $h\in H$ extending $f$.
\end{enumerate}
\end{definition}

Condition (i) is crucial in the proof of the abstract Ramsey theorem. We introduced the restrictions $\bullet$ and $\dbullet$ of pointwise
multiplication and pointwise action as in \eqref{E:setmul} and \eqref{E:seteval} to make sure that condition~(i) is fulfilled in concrete examples.
It says that $F\dbullet (G\dbullet S)$ is not defined by
chance; if it is defined, then the product $F\bullet G$ must be defined and it acts on $S$. Note that, by Lemma~\ref{L:aba}, this condition implies that
\[
F\dbullet (G\dbullet S)=(F\bullet G)\dbullet S.
\]
Condition (ii) is just a closure property.
As for condition (iii), in many situations, it would be useful to be able to deduce from $f\dowd \partial x$ being defined
that $f\dowd \partial x = \partial(f\dowd x)$. Such an implication fails as $f\dowd x$ may be undefined. To counter this failure,
condition (iii) ensures that if $f\dowd x$ is defined, for $f\in F$ and $x\in S$, then there is $h\in H$ such that
$h\dowd x$ is defined and
\[
f\dowd \partial x = \partial (h\dowd x),
\]
and that this happens for all $x\in S$ with the same $h$. We also point out that a careful reading of the proofs below shows that one can weaken 
the conclusion of condition (iii) to ``$H\dbullet S$ is defined and for each $f\in F$ there is $h\in H$ such that $h$ extends $f$ on $\partial S$, that is, for each 
$x\in \partial S$, $h\dowd x$ is defined and $h\dowd x = f\dowd x$." However, this refinement has not turned out to be useful so far.

\medskip

\noindent {\bf Example A4.}
We continue with Examples~A1--A3. We check that the set actoid defined in Example~A2 is a Ramsey domain.
Point (i) of the definition is clear since
\[
\binom{Q}{P}\dbullet (\binom{N}{M}\dbullet \binom{L}{K})
\]
is defined precisely when $M=L$ and $P=N$, in which case
\[
(\binom{Q}{P}\bullet \binom{N}{M})\dbullet \binom{L}{K}
\]
is defined as well. Let $\binom{L}{K}$ from the set actoid be given. Then $L\geq K\geq 1$ or $L=K=0$. If $K\geq 2$, then
$\partial \binom{L}{K} = \binom{L-1}{K-1}$; if $K=1$ or $K=0$, then
$\partial \binom{L}{K}=\binom{0}{0}$. Point (ii) of the definition follows immediately. To see point (iii),
assume that $\binom{N}{M}\dbullet \partial \binom{L}{K}$ is defined. If $K\geq 2$, then $M=L-1$ and $\binom{N+1}{L}$ witnesses that
point (iii) holds. If $K=1$, then $N=M=0$ and $\binom{L}{L}$ witnesses that (iii) holds. If $L=K=0$, then again
$N=M=0$ and $\binom{0}{0}$ witnesses (iii).

\medskip

\noindent {\bf Example B4.}
We continue with Examples~B1--B3. We check that the set actoid defined in Example~B2 is a Ramsey domain. Point (i) of the definition of Ramsey domain
is immediate from
the definition of $\bullet$ and $\dbullet$. Point (ii) is clear since $\partial S_K=S_K'$ and $\partial S_K'=S'_K$. To see point (iii), note that if
$F_{M,L}\dbullet \partial S_K$ or $F_{M,L}\dbullet \partial S'_K$ is defined, then $L=K$ and $F_{M,L}$ itself witnesses that point (iii) holds.

\begin{lemma}\label{L:voass}
Let $({\mathcal F},{\mathcal S})$ be a Ramsey domain. Let $S\in
{\mathcal S}$ and $F_1, \dots, F_n\in {\mathcal F}$. Assume that
\[
z_1 = F_n\dbullet (F_{n-1}\dbullet\cdots (F_2 \dbullet (F_1\dbullet
S)))
\]
is defined. Then
\[
z_2 = (F_n\bullet (F_{n-1}\cdots (F_2\bullet F_1)))\dbullet
S\;\hbox{ and }\; z_3= (((F_n\bullet F_{n-1})\cdots F_2)\bullet
F_1)\dbullet S
\]
are defined and $z_1=z_2=z_3$.
\end{lemma}

\begin{proof} One proves the existence of $z_2$ and $z_1=z_2$ and
the existence of $z_3$ and $z_1=z_3$ by separate inductions. To run
the inductive argument for $z_1=z_2$, note that by \eqref{E:acct}
and point (i) of Definition~\ref{D:ramdom}
\[
F_n\dbullet (F_{n-1}\dbullet\cdots (F_2\dbullet (F_1\dbullet S))) =
F_n\dbullet (F_{n-1}\dbullet\cdots (F_3\dbullet ((F_2\bullet
F_1)\dbullet S)))
\]
and apply the inductive assumption. Similarly, to run the induction
for $z_1=z_3$, note that by \eqref{E:acct} and point (i) of Definition~\ref{D:ramdom}
\[
F_n\dbullet (F_{n-1}\dbullet\cdots (F_2\dbullet (F_1\dbullet S)))
=(F_n\bullet F_{n-1})\dbullet (F_{n-2}\dbullet\cdots (F_2\dbullet
(F_1\dbullet S))),
\]
and apply the inductive assumption.
\end{proof}

\section{Ramsey and pigeonhole conditions and the first abstract Ramsey theorem}\label{S:RP}

\subsection{Ramsey condition (R) for set actoids}

At this point, we can state the abstract Ramsey property alluded to in the introduction. Note that
it can be stated for set actoids (truncation is not needed). So let $({\mathcal F}, {\mathcal S})$ be a set actoid.

\medskip
\noindent {\em {\bf Condition (R).} For each $d>0$ and each $S\in {\mathcal S}$ there exists
$F\in {\mathcal F}$ such that $F\dbullet S$ is defined, and for each
$d$-coloring of $F\dbullet S$ there exists $f\in F$ with $f\dowd S$
monochromatic.}
\medskip

Condition (R) from the definition above, when
interpreted for the set actoid from Example~A2 becomes just the
classical Ramsey theorem.

Condition (R) is a very abstract property. Essentially all finite Ramsey theorems, even those which we cannot prove
using the methods of this paper, like for example structural Ramsey theorems, can be seen as particular instances of (R).
(One should point out that possibilities of phrasing certain Ramsey statements using actions of {\em groups} have been explored
in \cite[Section 4]{KePeTo05}, \cite[Section 1.5]{Pe06}, and, most recently, in \cite[Section 4]{Bl11}.)
Of course, at this point, the main
problem is: are there general settings in which (R) can actually be proved?
The structure of set actoids over actoids is not rich enough to support such proofs. For this we will need
the structure of Ramsey domains over composition spaces and normed composition spaces. We will also need to formulate pigeonhole principles from
which condition (R) can be deduced and which can be formulated only using these richer structures. We proceed to the formulation of such
a pigeonhole principle right now.

\subsection{Pigeonhole principle for Ramsey domains}

We formulate two pigeonhole principles:
one here called (P) and a localized version of it called (LP) in
Section~\ref{S:loc}. They are not straightforward abstractions of
the classical Dirichlet's pigeonhole principle. Rather they are
conditions that make it possible to carry out inductive arguments
proving the Ramsey property, they are easy to verify in concrete
situations and are flexible enough to accommodate in applications many concrete statements as
special cases. For example, the abstract pigeonhole principle (LP) reduces to Dirichlet's pigeonhole principle
in the case of the classical Ramsey theorem; however, in different situations a variety of
other statements, like the Hales--Jewett theorem or the Graham--Rothschild
theorem serve as pigeonhole principles.

The pigeonhole principle (P) below can be thought of in the
following way. The Ramsey condition requires, upon coloring of
$F\dowd S$, fixing of a color on $f\dowd S$ for some $f\in F$. In
condition (P), we consider the equivalence relation on $S$ that
identifies $x_1$ and $x_2$ from $S$ if $\partial x_1 =
\partial x_2$. The pigeonhole principle (P) requires fixing of
a color on each equivalence class separately, rather than on the
whole $S$, after acting by an element of $F$.

Let $({\mathcal F}, {\mathcal S})$ be a Ramsey domain.

\medskip

\noindent {\em {\bf Condition (P).} For every $d>0$ and $S\in {\mathcal S}$
there exists $F\in {\mathcal F}$ such that $F\dbullet S$ is defined
and for each $d$-coloring $c$ of $F\dbullet S$ there
exists $f\in F$ such that for all $x_1, x_2\in S$ we have
\[
\partial x_1 = \partial x_2\Longrightarrow c(f\dowd x_1)=c(f\dowd x_2).
\]}

\medskip

It is convenient to illustrate the above definition by sequence B of
examples. The localized pigeonhole principle from
Section~\ref{S:loc} will be illustrated by sequence A.

\medskip

\noindent {\bf Example B5.}
Before we continue with Examples~B1--B4, we state Dirichlet's pigeonhole principle phrased here in
a surjective form.
\begin{enumerate}
\item [($*$)]For every $d>0$ and $K\geq 2$ there exists $L\geq 2$
such that for each $d$-coloring of all $q\colon [L]\to [2]$, $q\in
{\rm IS}$, there exists $q_0\colon [L]\to [K]$, $q_0\in {\rm IS}$,
such that
\[
\{ p\circ q_0\mid p\colon [K]\to [2],\, p\in {\rm IS}\}
\]
is monochromatic.
\end{enumerate}
One can take
\begin{equation}\notag
L= d(K-2)+2.
\end{equation}

We claim that the Ramsey domain $({\mathcal F}_0, {\mathcal
S}_{K_0})$ defined in Example~B2 and checked to be a Ramsey domain in Example~B4 fulfills (P). Let $S_K, S_K'\in {\mathcal S}_{K_0}$. Note
that for $h_1,h_2\in S_K'$, $\partial h_1=
\partial h_2$ implies $h_1=h_2$. Therefore, we only need to check
condition (P) for $S_K$. Note that if $h\in S_{K}$, then $\partial
h$ uniquely determines $h$ among functions in $S_K$ unless $h$ is of
the following form: for some $0< K_1<K$,
\begin{equation}\label{E:above}
h\upharpoonright [K_1] \equiv K_0\;\hbox{ and }\; h\upharpoonright
([K]\setminus [K_1]) \equiv \max(1, K_0- 1).
\end{equation}
It follows that given $d>0$, we need to find $L\geq K$ so that for
each $d$-coloring $c$ of $F_{L,K}\dbullet S_K$ there is $p\in
F_{L,K}$ such that the color $c(h\circ p)$ is constant for $h\in
S_K$ of the form \eqref{E:above} as $K_1$ runs over $[K-1]$. Such an
$L$ exists by the virtue of the basic pigeonhole principle ($*$)
stated above.

\subsection{Pigeonhole implies Ramsey}\label{S:main}

We continue to adhere to the following convention: the three
operations on a composition space $(A, Z)$ are denoted by $\cdot$, $\dowd$,
and $\partial$, respectively, while the operations on a Ramsey domain
over $(A, Z)$ are denoted by $\bullet$ and $\dbullet$. We
also use the notation set up in \eqref{E:setmul}, \eqref{E:seteval},
and \eqref{E:setprim}.

Theorem~\ref{T:main2} and Corollary~\ref{C:mainco} give general
Ramsey statements derived from the pigeonhole principle (P).
Corollary~\ref{C:mainco} is simpler to state than
Theorem~\ref{T:main2} and is all that is needed from this theorem in
most, but not all, situations.

\begin{theorem}\label{T:main2}
Let $({\mathcal F}, {\mathcal S})$ be a Ramsey domain fulfilling condition (P).
For $d>0$, $t\geq 0$, and $S\in {\mathcal S}$,
there exists $F\in {\mathcal F}$ such that $F\dbullet S$ is defined
and for each $d$-coloring $c$ of $F\dbullet S$ there exists $f \in
F$ such that for $x_1, x_2\in S$
\begin{equation}\label{E:imp2}
\partial^t x_1 = \partial^t x_2\Longrightarrow c(f\dowd x_1) = c(f\dowd x_2).
\end{equation}
\end{theorem}

\begin{proof} We will prove the conclusion of the theorem assuming that
{\em for every $d>0$, $t\geq 0$, and $S\in {\mathcal S}$ there is
$F\in {\mathcal F}$ such that $F\dbullet S$ is defined and for every
$d$-coloring $c$ of $\partial^{t} (F\dbullet S)$ there is $f\in F$
such that for $x_1, x_2\in S$}
\begin{equation}\label{E:tro}
\partial^{t+1} x_1 = \partial^{t+1} x_2 \Longrightarrow c(\partial^{t} (f\dowd x_1)) = c(\partial^{t} (f\dowd x_2)).
\end{equation}

Making this assumption is justified since it follows from condition (P) as shown by the following argument carried out by induction on $t$. For
$t=0$ the assumption above is simply (P). Now we go from $t$ to $t+1$.
Let $S\in {\mathcal S}$. Since $({\mathcal F}, {\mathcal S})$ is
a Ramsey domain, we have $\partial S\in {\mathcal S}$. So by the above assumption for $t$ there is $F\in {\mathcal F}$ such that
$F\dbullet \partial S$ is defined and for every $d$-coloring $c$ of $\partial^t(F\dbullet \partial S)$ there is $f\in F$ such that for $x_1,x_2\in \partial S$
\eqref{E:tro} holds. Now, again since $({\mathcal F}, {\mathcal S})$ is a Ramsey domain, there is $G\in {\mathcal F}$ such that
$G\dbullet S$ is defined and every element of $F$ is extended by an element of $G$. We claim that $G$ makes the assumption true for $t+1$.
Let $c$ be a $d$-coloring of
\[
\partial^{t+1}(G\dbullet S) = \partial^t(G\dowd \partial S).
\]
Since $F\dbullet \partial S\subseteq G\dowd \partial S$, $c$ gives a $d$-coloring of $\partial^t(F\dbullet \partial S)$. Thus, by our choice of $F$, there
is $f\in F$ such that for $x_1,x_2\in \partial S$ \eqref{E:tro} holds. Let $g\in G$ extend this $f$. Then for $y_1,y_2\in S$ with
$\partial^{t+2}y_1 = \partial^{t+2}y_2$,  we have $\partial^{t+1}(\partial y_1) = \partial^{t+1}(\partial y_2)$,
so
\[
c(\partial^{t}(f\dowd \partial y_1)) = c(\partial^{t}(f\dowd \partial y_2)).
\]
Since $f\dowd \partial y_1 = g\dowd \partial y_1$ and
$f\dowd \partial y_2 = g\dowd \partial y_2$,  it follows that
\[
c(\partial^{t+1}(g\dowd y_1)) =
c(\partial^{t}(g\dowd \partial y_1)) = c(\partial^{t}(g\dowd \partial y_2)) = c(\partial^{t+1}(g\dowd y_2)) ,
\]
as required.

Now, we prove the theorem making the assumption from the beginning of the proof.
Fix $d>0$. The argument is by induction on $t\geq 0$ for all $S\in
{\mathcal S}$. For $t=0$, the conclusion is clear since it requires
only that there be a non-empty $F\in {\mathcal F}$ with $F\dbullet
S$ defined, which is guaranteed by our assumption. Now we suppose
that the conclusion of the theorem holds for $t$ and we show it for
$t+1$. Apply our assumption stated at the beginning of the proof to
$d$, $t$, and $S$ obtaining $F_0\in {\mathcal F}$. Note that
$F_{0}\dbullet S\in {\mathcal S}$. Apply the inductive assumption
for $t$ to $F_{0}\dbullet S$ obtaining $F_1\in {\mathcal F}$. Note
that $F_{1}\dbullet (F_{0}\dbullet S)$ is defined, hence, since
$({\mathcal F}, {\mathcal S})$ is a Ramsey domain,
$(F_1\bullet F_0)\dbullet S$ is defined and, by Lemma~\ref{L:asset},
\begin{equation}\label{E:equalit}
(F_1\bullet F_0)\dbullet S = F_{1}\dbullet (F_{0}\dbullet S)
\end{equation}
Note that $F_1\bullet F_0 \in {\mathcal F}$, and we claim that it
works for $t+1$.

Let $c$ be a $d$-coloring of $(F_1\bullet F_0)\dbullet S$. By
\eqref{E:equalit}, we can consider it to be a coloring of
$F_{1}\dbullet (F_{0}\dbullet S)$. By the choice of $F_1$ there
exists $f_1\in F_1$ such that for $x, y\in S$ and $f, g\in F_0$,
\begin{equation}\label{E:mirac}
\partial^{t} (f\dowd x) = \partial^{t} (g\dowd y)
\Longrightarrow c(f_1\dowd (f\dowd x)) = c(f_1\dowd (g\dowd y)).
\end{equation}
Define a $d$-coloring ${\bar c}$ of $\partial^{t} (F_{0}\dbullet S)$
by letting for $f\in F_0$ and $x\in S$
\begin{equation}\label{E:mircont}
{\bar c}(\partial^{t} (f\dowd x)) = c(f_1\dowd (f\dowd x)).
\end{equation}
The coloring ${\bar c}$ is well-defined by \eqref{E:mirac}. By our
choice of $F_0$, there exists $f_0\in F_0$ such that for $x, y\in S$
\begin{equation}\label{E:miren}
\partial^{t+1} x = \partial^{t+1} y \Longrightarrow
{\bar c}(\partial^{t}(f_0\dowd x)) = {\bar c}(\partial^{t}(f_0\dowd
y)).
\end{equation}
Combining \eqref{E:miren} with \eqref{E:mircont}, we see that for
$x, y \in S$
\begin{equation}\notag
\partial^{t+1} x = \partial^{t+1} y \Longrightarrow c(f_1\dowd (f_0\dowd x))
= c(f_1\dowd (f_0\dowd y)).
\end{equation}
Now $f=f_1\cdot f_0$ is as required since by \eqref{E:equalit} and
Lemma~\ref{L:asset}, we have
\[
f_1\dowd (f_0\dowd x) = (f_1\cdot f_0)\dowd x \;\hbox{ and }\;
f_1\dowd (f_0\dowd y) = (f_1\cdot f_0)\dowd y,
\]
and the proof is completed.
\end{proof}

\begin{definition}
A Ramsey domain $({\mathcal F}, {\mathcal S})$ is called {\em vanishing}
if for every $S\in {\mathcal S}$ there is $t\in {\mathbb N}$ such that $\partial^t S$ consists of
one element.
\end{definition}

\begin{corollary}\label{C:mainco}
Let $({\mathcal F}, {\mathcal S})$ be a vanishing Ramsey domain.
If $({\mathcal F}, {\mathcal S})$ fulfills condition (P), then it fulfills condition (R).
\end{corollary}

\begin{proof}
The conclusion follows from Theorem~\ref{T:main2} since for each
$S\in {\mathcal S}$ there is $t\in {\mathbb N}$ with $\partial^t S$
having at most one element. For this $t$, the left hand side in
\eqref{E:imp2} holds for all $x_1, x_2\in S$.
\end{proof}

\section{Localizing the pigeonhole condition and the second abstract Ramsey theorem}\label{S:loc}

We formulate here a localized version (LP) of condition (P) and
prove in Theorem~\ref{T:hp} that, under mild assumptions, it implies
(P), making checking (P) much easier. Even though condition (LP)
can be stated for Ramsey domains over composition spaces, the proof of
Theorem~\ref{T:hp} requires introduction in Subsection~\ref{S:aux}
of a new piece of structure on a composition space, which is nevertheless
found in almost all concrete situations.

\subsection{Localized version (LP) of (P)}\label{Su:verhj}

One can think of condition (LP) in the following way. In condition
(P), we are given a coloring of $F\dbullet S$ and are
asked to find $f\in F$ making the coloring constant on each
equivalence class of the equivalence relation on $S$ that identifies $y_1, y_2\in S$ if
$\partial y_1=\partial y_2$.
Obviously, it is easier to fulfill the requirement of making the
coloring constant, by multiplying by some $f\in F$, on a single,
fixed equivalence class of this equivalence relation. Condition
(${\rm LP}$) makes just such a requirement. The price for
this weakening of the pigeonhole principle is paid by putting an additional
restriction on the element $f\in F$ fixing the color. We will
comment on this restriction after the condition is stated. First, we
introduce a piece of notation for equivalence classes of the
equivalence relation mentioned above. For $S\subseteq Z$ and $x_0\in
Z$, put
\[
S_{x_0} = \{ y\in S \mid \partial y = x_0\}.
\]
Also, for $F\subseteq A$, let
\[
F_a = \{ b\in F\mid b\hbox{ extends }a\}.
\]

Let $({\mathcal F}, {\mathcal S})$ be a Ramsey domain over a
composition space $(A, Z)$. The following criterion on $({\mathcal F},
{\mathcal S})$ turns out to be the right formalization of the local
version of (P).

\medskip

\noindent {\em {\bf Condition (LP).} For $d>0$, $S\in {\mathcal S}$,
and $x\in \partial S$, there is $F\in {\mathcal F}$ and $a\in
A$ such that $F\dbullet S$ is defined, $a\dowd x$ is defined, and
for every $d$-coloring of $F_a\dowd S_{x}$ there is
$f\in F_a$ such that $f\dowd S_{x}$ is monochromatic.}

\medskip

The equivalence relation on $S$ given by $\partial y_1 =
\partial y_2$ obviously has $\partial S$ as its
set of invariants, that is, two elements of $S$ are
equivalent if and only if their images in $\partial S$ under
the function $y\to \partial y$ are the same. In condition (LP), we consider the equivalence class given by $x\in
\partial S$ and we ask for $a\in A$ that acts on a part of
the set of invariants $\partial S$ including $x$ and is such
that each $d$-coloring can be stabilized on $S_x$ by
multiplication by some $f\in F$ that acts in a manner compatible
with $a$.

\subsection{Normed composition spaces}\label{S:aux}

We introduce here a new piece of structure on composition spaces.

\begin{definition} Let $(A,Z)$ be a composition space. We say that
$(A,Z)$ is {\em normed} if there is a function $|\cdot |\colon Z\to
D$, where $(D, \leq)$ is a partial order, such that for $x,y\in Z$,
$|x|\leq |y|$ implies that for all $a\in A$
\[
a\dowd y \hbox{ defined}\Rightarrow (a\dowd x\hbox{ defined and }
|a\dowd x|\leq |a\dowd y|).
\]
\end{definition}
A function $|\cdot |$ as in the above definition will be called a {\em norm}. In most
cases, for example in all cases considered in this paper, $(D,\leq)$ will be a linear order.
However, there are natural situations, occurring in a forthcoming work, in which $(D,\leq)$ is not linear. 
A remnant of linearity will always be retained as explained in Definition~\ref{D:linear} below.

\medskip

\noindent {\bf Example A5.}
We continue with Examples~A1--A4.
Define $|\cdot |\colon {\rm II}\to {\mathbb N}$ by
\begin{equation}\notag
|i| = \max {\rm range}(i),
\end{equation}
for $i\in {\rm II}$. It is easy to see that the function defined
above is a norm on $(B,Y)$; thus, $(B,Y)$ becomes a normed
composition space.

Checking that $({\mathcal F}, {\mathcal S})$ defined in Example~A2
fulfills (${\rm LP}$) amounts to an application of the standard
pigeonhole principle. Clearly $\mathcal S$ is vanishing. It follows from
Corollary~\ref{C:mainco2} below that $({\mathcal F}, {\mathcal S})$
is a Ramsey domain fulfilling (R), which gives the classical Ramsey theorem.

\subsection{Localized pigeonhole implies Ramsey}

We need one more definition.

\begin{definition}\label{D:linear}
A Ramsey domain $({\mathcal F}, {\mathcal S})$ over a normed composition space is called {\em linear} if
the image of $S$ under the norm is linear for each $S\in {\mathcal S}$.
\end{definition}

Here is the main theorem of this section.

\begin{theorem}\label{T:hp}
Let $({\mathcal F}, {\mathcal S})$ be a linear Ramsey domain over a normed composition space.
Assume that $\mathcal S$ consists of finite sets.
If $({\mathcal F}, {\mathcal S})$ fulfills ($\hbox{LP}$),
then $({\mathcal F}, {\mathcal S})$ fulfills (P).
\end{theorem}

\begin{proof} Let $(A,Z)$ with $\cdot,\, \dowd,\, \partial, \, |\cdot|$
be the normed composition space over which $({\mathcal F}, {\mathcal S})$
is defined. For the sake of clarity, in this proof, expressions of
the form
\[
F_kF_{k-1}\cdots F_1S\;\hbox{ and }\; f_kf_{k-1}\cdots f_1x
\]
stand for
\[
F_{k}\dbullet (F_{k-1}\dbullet\cdots (F_{1}\dbullet S))\;\hbox{ and
}\; f_k\dowd (f_{k-1}.\cdots (f_1\dowd x)),
\]
respectively. In particular, $fx$ stands for $f\dowd x$.

Fix $d>0$. Let $S\in {\mathcal S}$. Since
$\partial S$ is finite, $\partial S\in {\mathcal S}$ and $({\mathcal F}, {\mathcal S})$ is linear,
we can list $\partial S$ as $x_1, x_2, \dots, x_n$ with
\[
|x_n|\leq |x_{n-1}|\leq \cdots \leq |x_1|.
\]
We produce $F_1, \dots, F_n\in {\mathcal F}$ and $b_1, \dots, b_n\in
A$ as follows. For $1\leq k\leq n+1$, after $k-1$-st step of the
induction is completed, we have constructed $F_1, \dots, F_{k-1}$,
$b_1, \dots, b_{k-1}$. They have the following properties:
\begin{enumerate}
\item[(a)] $F_{k-1}F_{k-2}\cdots F_1S$ is defined;

\item[(b)] $b_{k-1}b_{k-2} \cdots b_1 x_{l}$ is defined for
$k-1\leq l\leq n$;

\item[(c)] for $1\leq j\leq k-1$, for every $d$-coloring of
$F_j\dowd (F_{j-1}\cdots F_1S)_{b_{j-1}\cdots b_1x_j}$ there is $f_j\in
F_j$ extending $b_j$ such that
\[
f_j\dowd (F_{j-1}\cdots F_1S)_{b_{j-1}\cdots b_1x_j}
\]
is monochromatic;

\item[(d)] for $1\leq j\leq k-1\leq l\leq n$, if $f_j\in F_j$ extends $b_j$
and ${\tilde x}\in S$ is such that $\partial {\tilde x} = x_l$, then
\begin{equation}\notag
\partial (f_{k-1}f_{k-2}\cdots f_1 {\tilde x}) = b_{k-1} b_{k-2} \cdots b_1x_{l}.
\end{equation}
\end{enumerate}
We make step $k\leq n$ of the recursion. With the fixed $d$,
we apply (LP) to $F_{k-1}F_{k-2}\cdots F_1S$, which exists
by (a) and obviously is in ${\mathcal S}$, and to $b_{k-1}b_{k-2}
\cdots b_1 x_{k}\in Z$, which exists by (b). This is permissible. Indeed, $x_k\in \partial S$ and,
by (c), there are $f_j\in F_j$ extending $b_j$ for $1\leq j\leq k-1$, and so, by
(d) taken with $l=k$, we get
\[
b_{k-1}\cdots b_1 x_k\in \partial (F_{k-1} F_{k-2} \cdots
F_1S).
\]
This application of ($\hbox{LP})$ gives $F_k\in {\mathcal F}$ and
$b_k\in A$. Now (a), (b), and (c) follow immediately from our choice
of $F_k$ and $b_k$ and the assumption $|x_{l}|\leq |x_{k}|$ for
$l\geq k$. Point (d) is a consequence of (a) and (b) for $k$ and (d)
for $k-1$ by the following argument. Fix $k\leq l\leq n$. Let
$f_j\in F_j$ extend $b_j$,
for each $1\leq j\leq k$, and let ${\tilde x}\in S$ be such that
$\partial {\tilde x} = x_l$. Note that using (d) for $k-1$ and
with the fixed $l$, we get
\begin{equation}\label{E:triv}
\partial (f_{k-1}\cdots f_1{\tilde x}) = b_{k-1} \cdots b_1 x_{l}.
\end{equation}
Thus, since, by (b) for $k$, $b_k b_{k-1} \cdots b_1 x_{l}$ is defined, so is $b_k
\partial (f_{k-1}\cdots f_1 {\tilde x})$. Now, since $f_k$
extends $b_k$, we see that $f_k \partial (f_{k-1}\cdots f_1{\tilde x})$ exists and
\begin{equation}\label{E:triv2}
f_k \partial (f_{k-1}\cdots f_1{\tilde x}) = b_k
\partial (f_{k-1}\cdots f_1 {\tilde x}).
\end{equation}
Putting \eqref{E:triv} and \eqref{E:triv2} together, we get (d) for
$k$ since
\begin{equation}\notag
\begin{split}
\partial (f_{k} f_{k-1}\cdots f_1 {\tilde x}) &=
f_k \partial (f_{k-1}\cdots f_1{\tilde x})\\
&= b_k \partial (f_{k-1}\cdots f_1 {\tilde x}) = b_k b_{k-1}
\cdots b_1 x_{l}.
\end{split}
\end{equation}
Note that above $f_{k} f_{k-1}\cdots f_1 {\tilde x}$ is defined by (a) for $k$.

So the recursive construction has been carried out. Note that by (a)
\begin{equation}\label{E:trtr}
F_{n}F_{n-1}\cdots F_1S
\end{equation}
is defined. We can apply Lemma~\ref{L:voass} to the Ramsey domain
$({\mathcal F}, {\mathcal S})$ to see that
$(F_{n}\bullet (F_{n-1} \bullet\cdots\bullet F_1))\dbullet S$ is
defined as well. Now, $F_{n}\bullet (F_{n-1} \bullet\cdots\bullet
F_1)$ is an element of ${\mathcal F}$, and we claim that for each
$d$-coloring $c$ of
\[
(F_{n}\bullet (F_{n-1} \bullet\cdots\bullet F_1))\dbullet S
\]
there are $f_1\in F_1, \dots, f_n\in F_n$ such that for $x_1, x_2\in S$ we have
\begin{equation}\label{E:dein}
\partial x_1= \partial x_2\Longrightarrow c((f_{n}\cdot (f_{n-1}
\cdots f_1))\dowd x_1)= c((f_{n}\cdot (f_{n-1} \cdots f_1))\dowd
x_2).
\end{equation}
This will verify that $({\mathcal F}, {\mathcal S})$ fulfills (P).

Fix, therefore, a $d$-coloring $c$ of $(F_{n}\bullet (F_{n-1}
\bullet\cdots\bullet F_1))\dbullet S$. We recursively
produce $f_n\in F_n, \dots , f_1\in F_1$. Note first that since
\eqref{E:trtr} is defined, by Lemmas~\ref{L:voass},
we have that for each $1\leq k\leq n$
\begin{equation}\notag
\begin{split}
&(F_{n}\bullet (F_{n-1} \bullet\cdots\bullet F_1))\dbullet S\\
&= F_{n}\dbullet (F_{n-1}\dbullet\cdots (F_k\dbullet (F_{k-1}F_{k-2}\cdots F_1S))).
\end{split}
\end{equation}
Therefore, having produced $f_n, \dots, f_{k+1}$, we can consider
the $d$-coloring of $F_k\dbullet (F_{k-1} F_{k-2} \cdots
F_1S)$ given by
\[
fy \to c(f_n\cdots f_{k+1}f y),
\]
for $f\in F_k$ and $y\in F_{k-1} F_{k-2}\cdots F_1S$. By (c), we get $f_k\in F_k$ such that
\begin{equation}\label{E:coco}
\begin{split}
c(f_n\cdots f_{k+1} f_k y)& \hbox{ is constant for }\\
&y\in (F_{k-1} F_{k-2}\cdots F_1S)_{b_{k-1}
b_{k-2}\cdots b_1 x_k}
\end{split}
\end{equation}
and $f_k$ extends $b_k$.

We claim that $f_n, \dots, f_1$ produced this way witness that
\eqref{E:dein} holds. Let $y_1, y_2\in S$ be such that
$\partial y_1=
\partial y_2$, and let this common value be $x_k$ for some $1\leq
k\leq n$. For $i=1,2$, condition (d) gives
\begin{equation}\notag
b_{k-1}b_{k-2}\cdots b_{1}x_k
=\partial (f_{k-1}f_{k-2}\cdots f_{1}y_i).
\end{equation}
and so
\[
f_{k-1} f_{k-2}\cdots f_{1}y_i \in (F_{k-1} F_{k-2}\cdots
F_1 S)_{b_{k-1} b_{k-2}\cdots b_1 x_k},
\]
which in light of \eqref{E:coco} implies that
\begin{equation}\label{E:coeq}
c(f_{n}\cdots f_{k+1} f_k (f_{k-1}\cdots f_{1}y_1)) = c(f_{n}\cdots
f_{k+1} f_k (f_{k-1}\cdots f_{1}y_2)).
\end{equation}
Since \eqref{E:trtr} is defined, by Lemma~\ref{L:voass}, applied to
the Ramsey domain $({\mathcal F}, {\mathcal S})$, and by inductively
applying Lemma~\ref{L:asset}, we get that for $i=1,2$,
\[
f_{n} f_{n-1}\cdots f_{1} y_i = (f_{n}\cdot (f_{n-1}\cdots f_{1})) y_i.
\]
From this equality and from \eqref{E:coeq} the conclusion follows.
\end{proof}

The following corollary is an immediate consequence of
Theorem~\ref{T:hp} and Corollary~\ref{C:mainco}.

\begin{corollary}\label{C:mainco2}
Let $({\mathcal F}, {\mathcal S})$ be a linear vanishing Ramsey domain over a
normed composition space. Assume that $\mathcal S$ consists of finite sets.
If $({\mathcal F}, {\mathcal S})$ fulfills ($\hbox{LP}$), then it fulfills (R).
\end{corollary}

\subsection{Remarks on normed composition spaces}

The most involved algebraic notion introduced in the paper is the notion
of normed composition space. Below, we give a list of conditions that are
more symmetric than those defining normed composition spaces. We then prove
in Lemma~\ref{L:back} that the new conditions define a structure
that is essentially equivalent to a normed composition space. It is worth
remarking that all the normed composition spaces in the present paper, in \cite{So3}, and in \cite{So4}
fulfill the conditions below.

Let $(A, Z, \cdot, \dowd, \partial, |\cdot|)$ be such that $\cdot$
is a partial function from $A\times A$ to $A$, $\dowd$ is a partial
function from $A\times Z$ to $Z$, $\partial$ is a function from $Z$
to $Z$ and $|\cdot |$ is a function from $Z$ to a set with a partial
order $\leq$. We consider the following set of conditions:
\begin{enumerate}
\item[(a)] if $a\dowd (b\dowd z)$ and $(a\cdot b) \dowd z$ are defined for $a,b\in A$ and $z\in Z$, then
\[
a\dowd (b\dowd z) = (a\cdot b) \dowd z;
\]

\item[(b)] if $a\dowd z$ and $a\dowd \partial z$ are defined for $a\in A$ and $z\in Z$, then
\[
\partial(a\dowd z) = a\dowd \partial z;
\]

\item[(c)] $|\partial z|\leq |z|$ for each $z\in Z$;

\item[(d)] if $|y|\leq |z|$ and  $a\dowd z$ is defined for $a\in A$ and $y, z\in Z$, then $a\dowd y$ is defined and
$|a\dowd y|\leq |a\dowd z|$.
\end{enumerate}
Condition (a) is just the action property. Conditions (b), (c) and (d) say that each pair of functions constituting
the structure interact in a  natural way: action with truncation in (b), truncation with norm in (c), and
norm with action in (d).

\begin{lemma}\label{L:back}
If $(A, Z, \cdot, \dowd, \partial, |\cdot|)$ fulfills conditions (a)--(d), then $(A,Z)$ with $\cdot$, $\dowd$, $\partial$
and $|\cdot|$ is a normed composition space
\end{lemma}

\begin{proof} Almost all the properties defining a normed composition space are already explicit among (a)--(d).
One only needs to check that for $a\in A$ and $z\in Z$ if $a\dowd z$
is defined, then so is $a\dowd\partial z$, and this property follows
from (c) and (d).
\end{proof}

\section{Propagating the pigeonhole principle}\label{S:prop}

In this section, we prove two results that make it possible to
propagate condition (P) to new examples. In the first result, we
show how to obtain condition (P) on appropriately defined finite
products assuming it holds on the factors. The second result
involves the notion of interpretation of sets from a Ramsey domain in
another Ramsey domain. This result shows that if each
set from a Ramsey domain is interpretable in some Ramsey domain
fulfilling (P) then that Ramsey domain fulfills (P) as well.

\subsection{Products} We prove here a consequence of
Theorem~\ref{T:main2} that extends this theorem to products. First,
we set up a general piece of notation. Let $X_i$, $1\leq i\leq l$,
be sets, and let ${\mathcal U}_i$ be a family of subsets of $X_i$.
Let
\[
\bigotimes_{i\leq l} {\mathcal U}_i = \{ \prod_{i=1}^l U_i\mid U_i\in
{\mathcal U}_i\hbox{ for }i=1, \dots, l\}.
\]
When ${\mathcal U}_i= {\mathcal U}$ for all $1\leq i\leq l$, we
write ${\mathcal U}^{\otimes l}$ for $\bigotimes_{i\leq l} {\mathcal U}_i$.
Note that $\bigotimes_{i\leq l} {\mathcal U}_i$ consists of subsets of
$\prod_{i\leq l}X_i$.

Let $(A_i, Z_i)$, $1\leq i\leq l$, be composition spaces. The
multiplication and action on each of them is denoted by the same symbols
$\cdot$ and $\dowd$; the truncation on $Z_i$ is denoted by $\partial_i$. The
product of $(A_i, Z_i)$, $1\leq i\leq l$, is defined in the natural
coordinatewise way. Its underlying sets are
\[
\prod_{i=1}^l A_i\;\hbox{ and }\; \prod_{i=1}^lZ_i;
\]
the multiplication $(a_i)\cdot (b_i)$, for $(a_i), (b_i)\in
\prod_{i\leq l} A_i$, is declared to be defined precisely when
$a_i\cdot b_i$ is defined for each $i\leq l$ and then
\[
(a_i)\cdot (b_i) = (a_i\cdot b_i);
\]
the action $(a_i)\dowd (z_i)$, for $(a_i)\in \prod_{i\leq l} A_i$
and $(z_i)\in \prod_{i\leq l} Z_i$, is defined precisely when
$a_i\dowd z_i$ is defined for each $i\leq l$ and then
\[
(a_i)\dowd (z_i) = (a_i\dowd z_i);
\]
the truncation $\partial_\pi$ of $(z_i)$ is given by
\[
\partial_{\pi}(z_i) = (\partial_iz_i).
\]
It is easy to check, and we leave it to the reader, that the definitions above describe a composition space. We call it the product
composition space $(\prod_{i\leq l}A_i, \prod_{i\leq l}Z_i)$. If $(A_i, Z_i) = (A,Z)$ for each $i\leq
l$, we write $(A^l, Z^l)$ for $(\prod_{i\leq l} A_i, \prod_{i\leq
l}Z_i)$.

Let now $({\mathcal F}_i, {\mathcal S}_i)$ be Ramsey domains over
$(A_i, Z_i)$, $1\leq i\leq l$. We define the operations on
$\bigotimes_{i\leq l} {\mathcal F}_i$, $\bigotimes_{i\leq l} {\mathcal S}_i$ as
follows. We declare $(\prod_{i\leq l} F_i)\bullet (\prod_{i\leq l}
G_i)$ to be defined precisely when $F_i\bullet G_i$ is defined for
each $1\leq i\leq l$ and then we let
\[
(\prod_{i=1}^l F_i)\bullet (\prod_{i=1}^l G_i) = \prod_{i=1}^l
(F_i\bullet G_i),
\]
and $(\prod_{i\leq l} F_i)\dbullet (\prod_{i\leq l}S_i)$ is defined
if $F_{i}\dbullet S_i$ is defined for each $i$ and then
\[
(\prod_{i=1}^l F_i)\dbullet (\prod_{i=1}^l S_i) = \prod_{i=1}^l
(F_{i}\dbullet S_i).
\]

\begin{lemma}\label{L:hep}
For $1\leq i\leq l$, let $({\mathcal F}_i, {\mathcal S}_i)$ be Ramsey domains over composition spaces
$(A_i, Z_i)$. Then $(\bigotimes_{i\leq l} {\mathcal F}_i,
\bigotimes_{i\leq l}{\mathcal S}_i)$ is a Ramsey domain
over the composition space $(\prod_{i\leq l}A_i, \prod_{i\leq l}Z_i)$.
\end{lemma}

\begin{proof} All the points from the definition of Ramsey domain are clear. For example, to check point
(iii), assume that $(\prod_{i\leq l} F_i)\dbullet \partial_\pi (\prod_{i\leq l} S_i)$ is defined
for some $\prod_{i\leq l} F_i \in \bigotimes_{i\leq l}{\mathcal F}_i$ and $\prod_{i\leq l} S_i\in \bigotimes_{i\leq l} {\mathcal S}_i$.
This means
that for each $1\leq i\leq l$, $F_i\dbullet \partial_i S_i$ is defined, so we can find $H_i\in {\mathcal F}_i$
as required by point (iii) of the definition of Ramsey domain for $({\mathcal F}_i, {\mathcal S}_i)$. Then clearly point (iii)
holds for $\prod_{i\leq l} H_i\in \bigotimes_{i\leq l}{\mathcal F}_i$.
\end{proof}

The following proposition propagates the pigeonhole principle from
factors to products. In the proof of this proposition
Theorem~\ref{T:main2} is used even though we are checking only condition (P).

\begin{proposition}\label{P:prod}
Let $({\mathcal F}_i, {\mathcal S}_i)$, $1\leq i\leq l$, be Ramsey domains fulfilling (P).
Assume that each ${\mathcal S}_i$
consists of finite sets. Then $(\bigotimes_{i\leq l}{\mathcal F}_i, \bigotimes_{i\leq l}{\mathcal S}_i)$ is a Ramsey domain fulfilling (P).
\end{proposition}

\begin{proof} By Lemma~\ref{L:hep}, it suffices to check (P).

We define a composition space structure on
\[
A_*=\prod_{i=1}^l A_i,\; Z_*= \{ 0, \dots,
l-1\}\times\prod_{i=1}^lZ_i
\]
as follows. The multiplication on $A_*$ is the same as in the
product composition space $(\prod_{i\leq l}A_i, \prod_{i\leq l}Z_i)$. For
$(a_i)\in A_*$ and $(p,(z_i))\in Z_*$, we make $(a_i)\dowd
(p,(z_i))$ be defined if $a_i\dowd z_i$ is defined for all $i$ and
\[
(a_i)\dowd (p,(z_i)) = (p, (a_i\dowd z_i)).
\]
For $(p, (z_i))\in Z_*$, let
\[
\partial_*(p, (z_i)) = (p+ 1 \,({\rm mod }\, l), (y_i)),
\]
where $y_i = z_i$ if $i\not= p+1$ and $y_{p+1} = \partial_{p+1}z_{p+1}$.
It is easy to see that $(A_*, Z_*)$ is a composition space.

Define
\[
{\mathcal F}_* = \bigotimes_{i\leq l} {\mathcal F}_i
\]
and let ${\mathcal S}_*$ consist of all sets of the form
\[
\{ p\}\times S,
\]
where $p\in \{ 0, \dots , l-1\}$, and $S\in \bigotimes_{i\leq l} {\mathcal
S}_i$. Define $\bullet$ on ${\mathcal F}_*$ to coincide with
$\bullet$ on $\bigotimes_{i\leq l} {\mathcal F}_i$. Declare $F\dbullet (\{
p\}\times S)$ to be defined if and only if $F\dbullet S$ is defined
in $(\bigotimes_{i\leq l} {\mathcal F}_i, \bigotimes_{i\leq l} {\mathcal S}_i)$, and
let
\[
F\dbullet (\{ p\}\times S) = \{ p\}\times (F\dbullet S).
\]
It is easy to check that $({\mathcal F}_*, {\mathcal S}_*)$ is a Ramsey domain
over the composition space $(A_*, Z_*)$.

We claim that $({\mathcal F}_*, {\mathcal S}_*)$ fulfills (P). To prove it, fix $d>0$ and
\[
\{ p\}\times \prod_{i=1}^l S_i \in {\mathcal S}_*,
\]
for some $S_i\in {\mathcal S}_i$. For $i\not= p+1$, pick $F_i\in {\mathcal F}_i$ such that
$F_{i}\dbullet S_i$ is defined. Such $F_i$ exists by condition (P)
with the number of colors equal to $1$ for the Ramsey domains $({\mathcal
F}_i, {\mathcal S}_i)$. Now, we apply condition (P) to $({\mathcal
F}_{p+1}, {\mathcal S}_{p+1})$ with the following number of colors:
\begin{equation}\label{E:colnum}
d^{\prod_{i\not= p+1}|F_i\dbullet S_i|}.
\end{equation}
(Note that the number defined above is finite since $F_{i}\dbullet
S_i$ is finite, as it belongs to ${\mathcal S}_i$.) This application gives us
$F_{p+1}\in {\mathcal F}_{p+1}$ such that $F_{p+1}\dbullet S_{p+1}$
is defined and for each coloring of $F_{p+1}\dbullet S_{p+1}$ with the number of colors given
by \eqref{E:colnum} there is $f\in F_{p+1}$ such that for any two
$x, y \in S_{p+1}$ fulfilling
\begin{equation}\label{E:coord}
\partial_{p+1}x = \partial_{p+1}y,
\end{equation}
$f\dowd x$ and $f\dowd y$ get the same color. Having defined $F_i$,
$1\leq i\leq l$, note that
\[
\prod_{i=1}^l F_i\in {\mathcal F}_*,
\]
and that
\[
(\prod_{i=1}^l F_i)\dbullet (\{ p\}\times \prod_{i=1}^lS_i)
\]
is defined. Given a $d$-coloring $c$ of
\begin{equation}\notag
(\prod_{i=1}^l F_i)\dbullet (\{ p\}\times \prod_{i=1}^l S_i),
\end{equation}
which set is equal to
\begin{equation}\notag
\{ p\}\times F_1\dbullet S_1\times \cdots \times
F_{p+1}\dbullet S_{p+1}\times\cdots\times F_l\dbullet S_l,
\end{equation}
consider the coloring of $F_{p+1}\dbullet S_{p+1}$ defined by
\begin{equation}\label{E:colbig}
h\to (c(p,h_1, \dots , h_p, h, h_{p+2}, \dots, h_l)\mid
(h_i)_{i\not= p+1}\in \prod_{i\not= p+1} F_i\dbullet S_i).
\end{equation}
This is a coloring with the number of colors equal to
\eqref{E:colnum}. Therefore, there exists $f_{p+1}\in F_{p+1}$ such
that for any two $x, y \in S_{p+1}$ fulfilling \eqref{E:coord}, $f_{p+1}\dowd x$
and $f_{p+1}\dowd y$ get the same color. Pick $f_i\in F_i$ for
$i\not= p+1$ arbitrarily. With these choices $(f_i)$ is an element
of $\prod_{i\leq l} F_i$. Note now that for
\[
(p, (x_i)),\, (p, (y_i)) \in \{ p\}\times \prod_{i=1}^l S_i
\]
we have
\begin{equation}\label{E:dereq}
\partial_*(p, (x_i)) =\partial_*(p, (y_i))
\end{equation}
precisely when $x_i = y_i$ for $i\not= p+1$ and \eqref{E:coord}
holds for $x_{p+1}$ and $y_{p+1}$. This observation allows us to say
that the definition of the coloring in \eqref{E:colbig} and our
choice of $f_{p+1}$ imply that if \eqref{E:dereq} holds, then
\[
c((f_i)\dowd (p, (x_i))) = c((f_i)\dowd (p, (y_i))).
\]
Thus, indeed, $({\mathcal F}_*, {\mathcal S}_*)$ is a Ramsey domain with (P).

Now apply Theorem~\ref{T:main2} (with $t=l$) to the Ramsey domain
$({\mathcal F}_*, {\mathcal S}_*)$, which has (P), while keeping in
mind that $\bigotimes_{i\leq l} {\mathcal F}_i= {\mathcal F}_*$ and that for $x\in \prod_{i\leq l}Z_i$, we have
\[
(0, \partial_\pi x) = \partial_*^{ l} (0, x).
\]
The proposition follows.
\end{proof}

\noindent {\bf Example B6.}
We continue with Examples~B1--B5.
Let $l\in {\mathbb N}$. By Proposition~\ref{P:prod} and Example~B5, the Ramsey domain
$({\mathcal F}_0^{\otimes l}, {\mathcal S}_{K_0}^{\otimes l})$ over the product composition space $({\rm
IS}^l, X_{K_0}^l)$ fulfills (P). This fact will be used in Subsection~\ref{Su:HJ}.

\subsection{Interpretations}

We introduce here a notion of interpretability. Its importance is contained in Proposition~\ref{P:inter}. It is a
general notion and we will need its full generality when proving the Hales--Jewett theorem.

\begin{definition} Let $({\mathcal F}, {\mathcal R})$ and $({\mathcal G}, {\mathcal S})$
be Ramsey domains over composition spaces $(A, X)$ and $(B, Y)$, respectively.
Let $S\in {\mathcal S}$. We say that $S$
is {\em interpretable in} $({\mathcal F}, {\mathcal R})$ if there
exists $R\in {\mathcal R}$ and a function $\alpha\colon S\to R$ such that
\begin{enumerate}
\item[(i)] for $y_1, y_2\in S$,
\[
\partial y_1 = \partial y_2\Longrightarrow \partial\alpha(y_1) =
\partial\alpha(y_2);
\]

\item[(ii)] if $F\dbullet R$ is defined for some $F\in {\mathcal
F}$, then there exist $G\in {\mathcal G}$, with $G\dbullet S$
defined, and a function $\phi\colon F\to G$ such that for $f_1,
f_2\in F$ and $y_1,y_2\in S$
\begin{equation}\label{E:inter}
f_1\dowd \alpha(y_1) = f_2\dowd \alpha(y_2)\Longrightarrow
\phi(f_1)\dowd y_1 = \phi(f_2)\dowd y_2.
\end{equation}
\end{enumerate}
\end{definition}

\begin{proposition}\label{P:inter}
Let $({\mathcal G}, {\mathcal S})$ be a Ramsey domain. If each $S\in {\mathcal S}$ is
interpretable in a Ramsey domain fulfilling condition (P), then $({\mathcal G},
{\mathcal S})$ fulfills condition (P).
\end{proposition}

\begin{proof}
Let $S\in {\mathcal S}$ and $d>0$ be given. Let
$({\mathcal F}, {\mathcal R})$ be a Ramsey domain with (P) over
$(A, X)$ in which $S$ is interpretable. Find $R\in {\mathcal
R}$ and $\alpha\colon S\to R$ as in the definition of interpretability. Since
$({\mathcal F}, {\mathcal R})$ fulfills (P), we can find $F\in
{\mathcal F}$ such that $F\dbullet R$ is defined and for each
$d$-coloring $c'$ of $F\dbullet R$ there exists $f\in F$
such that for all $x_1, x_2\in S$ we have
\begin{equation}\label{E:colpr}
\partial x_1 = \partial x_2\Longrightarrow c'(f\dowd x_1) = c'(f\dowd x_2).
\end{equation}
For $F$ given above, find $G\in {\mathcal G}$ such that $G\dbullet
S$ is defined and $\phi\colon F\to G$ for which \eqref{E:inter}
holds. Assume we have a $d$-coloring $c$ of $G\dbullet S$.
Define a $d$-coloring $c'$ of $F\dbullet R$ as an
arbitrary extension to $F\dbullet R$ of the function given by
\[
c'(f\dowd \alpha(y)) = c(\phi(f)\dowd y),
\]
where $f\in F$ and $y\in S$. Note that $c'$ is well
defined by \eqref{E:inter}. For this $c'$, find $f\in F$ for which
\eqref{E:colpr} holds. Let now $y_1, y_2\in T$ be such that
\begin{equation}\label{E:stap}
\partial y_1=\partial y_2.
\end{equation}
Since condition (i) in the definition of interpretability holds for
$\alpha$, \eqref{E:stap} gives
\[
\partial \alpha(y_1) = \partial\alpha(y_2).
\]
Therefore, by the definition of $c'$, by the choice of $f$ and since
$\alpha(y_1), \alpha(y_2)\in S$, we get
\[
c(\phi(f)\dowd y_1) = c'(f\dowd \alpha(y_1)) = c'(f\dowd
\alpha(y_2))= c(\phi(f)\dowd y_2).
\]
Thus, we see that \eqref{E:stap} implies the above equality. It follows
that $\phi(f)\in G$ is as required by condition (P) for the
coloring $c$.
\end{proof}

The proposition above will be applied in Section~\ref{S:appl}.

\section{Examples of composition spaces and Ramsey domains}\label{S:coram}

The remainder of the paper illustrates the theoretical results of the earlier sections. It
contains applications of the general results proved so far to
particular cases. For the most part, these applications involve only formulating new notions
and interpreting
some statements as other statements.

\subsection{Examples of truncations}\label{Su:extrnc}

\subsubsection{Forgetful truncation of rigid surjections}

The first type of truncation we introduce is obtained by forgetting
the largest value of a function. It is defined on rigid surjections.
We call it the {\em forgetful truncation} and we define it as
follows. Let $s\colon [L]\to [K]$ be a rigid surjection. If $K>0$,
then $L>0$, and we let
\[
L_0 = \min \{ y\in [L]\mid s(y) = K\}.
\]
Define
\begin{equation}\label{E:tr1a}
\partial_f s = s\upharpoonright [L_0-1].
\end{equation}
If $K=0$, then $L=0$ and $s$ is the empty function, and we let
\begin{equation}\label{E:tr1b}
\partial_f \emptyset = \emptyset.
\end{equation}
Thus, unless $s$ is empty, $\partial_f s$ forgets the largest value
of $s$ while remaining a rigid surjection. Unless $s$ is empty,
$\partial_f s$ is a proper restriction of $s$.

\subsubsection{Confused truncation of surjections}

Another way of truncating a surjection is obtained by confusing the
largest value with the one directly below it. This type of
truncation is defined on non-empty surjections. We define the {\em
confused truncation} as follows. Let $v\colon [L]\to [K]$ be a
surjection with $K>0$. Define for $y\in [L]$
\begin{equation}\label{E:tr2a}
(\partial_{c}v)(y) =
\begin{cases}
v(y), &\text{ if $v(y)< K$;}\\
\max(1, K-1), &\text{ if $v(y)= K$.}
\end{cases}
\end{equation}
Note that $\partial_c v\colon [L]\to [\max(1, K-1)]$ is a
surjection. The confused truncation when applied to a non-empty
rigid surjection gives a rigid surjection.

\subsection{Examples of composition spaces}\label{Su:exba}

Three composition spaces were defined in Examples~A3, B3, and B6.
In this section, we describe a number of new composition spaces. They are
used in the proofs of Ramsey-type results later on. One more
composition spaces are given in Section~\ref{S:wal}.

\medskip

\noindent {\bf Normed composition space $(A_1, X_1)$.} Let $A_1={\rm RS}$
and
\[
X_1= \{ v\in {\rm S}\mid v\not=\emptyset\}.
\]
For $s_1, s_2\in A_1$ let $s_1\cdot s_2$ be defined when the
canonical composition $s_2\circ s_1$ is defined, and let
\[
s_1\cdot s_2 = s_2\circ s_1.
\]
For $s\in A_1$ and $v\in X_1$, let $s\dowd v$ be defined precisely
when the canonical composition $v\circ s$ is defined and let
\[
s\dowd v = v\circ s.
\]
We equip $X_1$ with the confused truncation $\partial_c$ given by
\eqref{E:tr2a}. Define a norm $|\cdot |\colon X_1\to {\mathbb N}$ by
letting $|v| = L$ for $v\in X_1$ with $v\colon [L]\to [K]$.

The following lemma is straightforward to prove.
\begin{lemma}\label{L:11}
$(A_1, X_1)$ with the operations defined above is a normed
composition space.
\end{lemma}

\medskip

\noindent {\bf Normed composition space $(A_2, X_2)$.} Let $A_2= X_2={\rm
RS}$. We define the multiplication on $A_2$ by the same formula
\[
s_2\cdot s_1 = s_1\circ s_2,
\]
for $s_1, s_2\in A_2$, and the action of $A_2$ on $X_2$ by
\[
t\dowd s = s\circ t,
\]
for $t\in A_2$ and $s\in X_2$, where all the compositions are
canonical compositions of rigid surjections and they are taken under
the assumptions under which canonical composition is defined. We
equip $X_2$ with the forgetful truncation $\partial_f$ given by
\eqref{E:tr1a} and \eqref{E:tr1b}. Define $|\cdot |\colon X_2\to
{\mathbb N}$ by letting for $t\colon [L]\to [K]$, $|t| = L$.

The following lemma is again straightforward to prove.
\begin{lemma}\label{L:22}
$(A_2, X_2)$ with the operations defined above is a normed
composition space.
\end{lemma}

\medskip

\noindent {\bf Normed composition space $(A_3, X_3)$.} Let $A_3 = X_3 =
{\rm AS}$. Given $(s,p), (t,q)\in {\rm AS}$, $s,p\colon [L]\to [K]$ and
$t,q\colon [N]\to [M]$, we let $(t,q)\cdot (s,p)$ and $(t,q)\dowd
(s,p)$ be defined if $M\geq L$ and in that case we let
\begin{equation}\notag
(t,q)\cdot (s,p) = (t,q)\dowd (s,p) = ((s\circ t)\upharpoonright
{\rm dom}(p\circ q),\, p\circ q),
\end{equation}
where all $\circ$ on the right hand side are canonical compositions,
and the left hand side is defined under the conditions under which
the canonical compositions on the right hand side are defined. We
also define a truncation $\partial$ on $X_3$ by
\begin{equation}\label{E:traug}
\partial (s,p) = (s\upharpoonright {\rm dom}(\partial_f p),\, \partial_f p),
\end{equation}
where $\partial_f$ is the forgetful truncation. Furthermore, we
define $|\cdot |\colon {\rm AS}\to {\mathbb N}$ by
\[
|(s,p)| = L
\]
if $s,p\colon [L]\to [K]$. We leave checking the following easy
lemma to the reader.

\begin{lemma}\label{L:as}
$(A_3, X_3)$ with the operations defined above is a normed
composition space.
\end{lemma}

\subsection{Examples of Ramsey domains}\label{Su:exact}

We give here examples of Ramsey domains that are relevant to
further considerations. Three Ramsey domains were already defined
in Examples~A2, B2, and B6, and one more will be defined in Section~\ref{S:wal}.

\medskip

\noindent {\bf Ramsey domain $({\mathcal F}_1, {\mathcal S}_1)$ over $(A_1,
X_1)$.} The family ${\mathcal F}_1$ consists of all sets of the form
\begin{equation}\label{E:fuf}
F_{N,M, L_0} = \{ s\in {\rm RS}\mid s\colon [N]\to [M],\;
s\upharpoonright [L_0] = {\rm id}_{[L_0]}\} ,
\end{equation}
for $0< L_0\leq M\leq N$. The family ${\mathcal S}_1$ consists of
sets of the form $S_{L, v_0}$ that are defined as follows. Let
$v_0\colon [L_0]\to [K_0]$ be a surjection for some $0< K_0\leq L_0$
and let $L_0\leq L$. Put
\[
S_{L, v_0} = \{ v\in {\rm S}\mid v\colon [L']\to [K_0]\hbox{ for
some } L\geq L'\geq L_0,\hbox{ and } v\upharpoonright [L_0] = v_0\}.
\]

We let $F_{Q,P,N_0}\bullet F_{M,L,K_0}$ be defined if $K_0=N_0$ and
$P=M$ and in that case we put
\[
F_{Q,M,K_0}\bullet F_{M,L,K_0}= F_{Q,L,K_0}.
\]
We let $F_{P, N, M_0}\dbullet S_{L,v_0}$ be defined, where
$v_0\colon [L_0]\to [K_0]$, if $L_0=M_0$ and $L=N$ and in that case
we set
\[
F_{P, L, L_0}\dbullet S_{L,v_0} = S_{P, v_0}.
\]
We leave checking the following easy lemma to the reader.
\begin{lemma}\label{L:actf1}
\begin{enumerate}
\item[(i)] $\partial_c S_{L,v_0} = S_{L, \partial_c v_0}$.

\item[(ii)] $({\mathcal F}_1, {\mathcal S}_1)$ is a Ramsey domain
over $(A_1, X_1)$.
\end{enumerate}
\end{lemma}

\medskip

\noindent {\bf Ramsey domain $({\mathcal F}_2, {\mathcal S}_2)$ over $(A_2, X_2)$.}
Both ${\mathcal F}_2$ and ${\mathcal S}_2$ consist of two types of sets. We start with
defining ${\mathcal F}_2$. Its elements are sets $F_{N,M, L_0}$ given by \eqref{E:fuf} and
sets
\[
G_{N,M, L_0} = \{ s\in {\rm RS}\mid s\colon [N']\to [M]\hbox{ for } L_0\leq N'\leq N,\; s\upharpoonright [L_0] = {\rm id}_{[L_0]}\}
\]
for $0<L_0\leq M\leq N$. Also let
\[
G_{0,0,0} = \{ \emptyset\}.
\]
The operation $\bullet$ on ${\mathcal F}_2$ is defined only in the following situations:
\[
F_{N,M,K_0}\bullet F_{M,L,K_0}\hbox{ and } G_{N,M,K_0}\bullet G_{M,L,K_0},
\]
and we let
\[
F_{N,M,K_0}\bullet F_{M,L,K_0}= F_{N,L,K_0}\hbox{ and }G_{N,M,K_0}\bullet G_{M,L,K_0}= G_{N,L,K_0}.
\]

Define the family ${\mathcal S}_2$ to consist of all sets of the
following two forms. Let $s_0\colon [L_0]\to [K_0]$ be a rigid surjection,
for some $L_0\geq K_0>0$, and let $K\geq K_0$, $L\geq L_0$, and $L\geq K$. Put
\begin{equation}\notag
S_{L,K, s_0} = \{ s\in {\rm RS}\mid s\colon [L]\to [K] \hbox{ and }
s\upharpoonright [L_0] = s_0\},
\end{equation}
and
\begin{equation}\notag
T_{L,K, s_0} = \{ s\in {\rm RS}\mid s\colon [L']\to [K]\hbox{ for
some } L\geq L'\geq L_0,\hbox{ and } s\upharpoonright [L_0] = s_0\}.
\end{equation}
Put also
\[
T_{0,0,\emptyset} = \{ \emptyset\}.
\]
Let the operation $\dbullet$ be defined only in the following two situations:
\[
F_{M, L, L_0}\dbullet S_{L,K,s_0} \hbox{ and }G_{M, L, L_0}\dbullet T_{L,K,s_0},
\]
where $s_0\colon [L_0]\to [K_0]$. In these situations, we let
\[
F_{M, L, L_0}\dbullet S_{L,K,s_0} = S_{M,K s_0}\hbox{ and }G_{M, L, L_0}\dbullet T_{L,K,s_0} = T_{M,K, s_0}.
\]

The proof of the following lemma amounts to easy checking. We leave it to the reader. In point (i) of this lemma,
parameters $L,K,s_0$ vary over the values for which the appropriate sets ($S_{L,K,s_0}$ and $T_{L,K,s_0}$) are defined.
Also in connection with the closure of ${\mathcal S}_2$ under truncation, note that for each rigid surjection $t\colon [L]\to [K]$, $\{ t\} = T_{L,K,t}$.

\begin{lemma}\label{L:g22}
\begin{enumerate}
\item[(i)] Let $[K_0]$ be the image of $s_0$. Then
\begin{equation}\notag
\begin{split}
&\partial_f S_{L,K,s_0}= T_{L-1, K-1, s_0}\;\hbox{ if } K>K_0>0;\;\;
\partial_f S_{L,K_0,s_0} = \{ \partial_f s_0\};\\
&\partial_f T_{L,K,s_0}= T_{L-1, K-1, s_0}\;\hbox{ if }
K>K_0>0;\;\;\partial_f T_{L,K_0,s_0} = \{ \partial_f s_0\}.
\end{split}
\end{equation}

\item[(ii)] $({\mathcal F}_2, {\mathcal S}_2)$ is a Ramsey domain over $(A_2, X_2)$.
\end{enumerate}
\end{lemma}

\medskip

\noindent {\bf Ramsey domain $({\mathcal F}_3, {\mathcal S}_3)$ over $(A_3, X_3)$.} In the
definitions below, we are slightly less general than in the
definitions of Ramsey domains described so far. As before, the
Ramsey domains consist of sets of elements of $A_3$ and $X_3$ that
map a given $[L]$ or its initial segment to a given $[K]$, but we refrain from considering
such sets with the additional requirement that elements in them
start with a fixed augmented surjection. This additional generality
can be easily achieved, but it is not needed in our application.

For $L\geq K>1$ or $L=K=1$, let
\begin{equation}\notag
F_{L,K} = \{ (s,p)\in {\rm AS}\mid s,p\colon [L]\to [K] \hbox{ and
} s^{-1}(K) = \{ L\}\},
\end{equation}
and, for $L\geq K>0$ or $L=K=0$, let
\begin{equation}\notag
G_{L,K} = \{ (s,p)\in {\rm AS}\mid s,p\colon [L']\to [K] \hbox{
for some }0< L'\leq L \}.
\end{equation}
Note that $ G_{0,0} = \{ (\emptyset, \emptyset)\}$.
Let ${\mathcal F}_3={\mathcal S}_3$ consist of all these sets defined above.

The two operations $\bullet$ and $\dbullet$ are equal to each other and they are defined
precisely in the situations described below and with the results specified below:
\[
F_{M,L}\bullet F_{L,K} = F_{M,L}\dbullet F_{L,K} = F_{M,K}.
\]
and
\[
G_{M,L}\bullet G_{L,K} = G_{M,L}\dbullet G_{L,K} = G_{M,K}.
\]

Recall definition \eqref{E:traug} of the truncation $\partial$ on
$(A_3, X_3)$. The following lemma is straightforward to check.
\begin{lemma}\label{L:SD}
\begin{enumerate}
\item[(i)] For $1<K\leq L$,
\[
\partial F_{L,K} = G_{L- 1, K- 1}\;\hbox{ and }\;
\partial G_{L,K} = G_{L- 1, K- 1},
\]
and, for $1\leq L$,
\[
\partial F_{0,0} = G_{0,0}\;\hbox{ and }\;\partial G_{L,1} = \partial G_{0,0} = G_{0,0}.
\]

\item[(ii)] $({\mathcal F}_3, {\mathcal S}_3)$ with the operations defined above is a Ramsey domain over $(A_3, X_3)$.
\end{enumerate}
\end{lemma}

\section{Applications}\label{S:appl}

In this section, we give applications of the methods developed in
the paper. We give two proofs in detail, that of the Hales--Jewett
theorem, in Subsection~\ref{Su:HJ}, and that of the self-dual Ramsey
theorem, in Subsection~\ref{Su:self}, as these two proofs are of
more interest than the other ones. These two proofs illustrate how
the results of Sections~\ref{S:loc} and \ref{S:prop} can be applied:
the proof of the Hales--Jewett theorem uses
Propositions~\ref{P:prod} and \ref{P:inter}, the proof of the
self-dual Ramsey theorem uses Theorem~\ref{T:hp}. Additionally, in
Subsection~\ref{Su:GR}, we sketch how to obtain the
Graham--Rothschild theorem and, in Section~\ref{S:wal}, we
describe a limiting case that is related to the considerations of
\cite{IrSo05}.

We prove these theorems here in the form in which they are stated in Subsection~\ref{Su:refor}, that is,
in terms of injections and surjections. The statements in terms of parameter sets, partitions and subsets
are given in Subsection~\ref{Su:self}.  In Subsection~\ref{Su:transl}, we described the way of
translating these statements from one form to the other.

\subsection{The Hales--Jewett theorem}\label{Su:HJ}

We prove below a theorem that combines into one the
usual Hales--Jewett theorem \cite{HeJe63} and Voigt's version of
this theorem for partial functions \cite[Theorem 2.7]{Vo80} as phrased in Subsection~\ref{Su:refor}. One
gets the classical Hales--Jewett theorem from the statement below
by setting $L=L_0+1$, $L_0=K_0$, and $v_0 = {\rm id}_{[K_0]}$ in the
assumption and $L'=L$ in the conclusion. Indeed, with the notation as in the statement below,
if the domains of both $v_1$ and $v_2$ are equal to $[L]$, then $v_1\circ s_0$ and $v_2\circ s_0$ get the
same color. One derives the Voigt
version for the same values of the parameters in the assumption and
for $L'<L$ in the conclusion. Indeed, the color of $v\circ s_0$ for $v\colon [L']\to [K_0]$ depends
only on $L'$, so one application of Dirichlet's pigeonhole principle gives Voigt's version of the Hales--Jewett
theorem.

\medskip

\noindent {\bf Hales--Jewett, combined version.} {\em Given $d>0$,
$0<K_0\leq L_0\leq L$ and a surjection $v_0\colon [L_0]\to [K_0]$,
there exists $M\geq L_0$ with the following property. For each
$d$-coloring $c$ of
\[
\{ v\colon [M']\to [K_0]\mid L_0\leq M'\leq M\hbox{ and }
v\upharpoonright [L_0] = v_0\}
\]
there exists a rigid surjection $s_0\colon [M]\to [L]$ such that
$s_0\upharpoonright [L_0] ={\rm id}_{[L_0]}$ and
\[
c(v_1\circ s_0) = c(v_2\circ s_0)
\]
whenever $v_1, v_2\colon [L']\to [K_0]$, for the same $L_0\leq
L'\leq L$, and $v_1\upharpoonright [L_0] = v_2\upharpoonright [L_0]=
v_0$.}

\medskip

We will use the Ramsey domain $({\mathcal F}_1, {\mathcal S}_1)$ defined in Subsection~\ref{Su:exact} and proved to be
a Ramsey domain in Lemma~\ref{L:actf1}.

The proof of the following lemma is an application of the notion of interpretability.

\begin{lemma}\label{L:preHJ}
$({\mathcal F}_1, {\mathcal S}_1)$ fulfills condition (P).
\end{lemma}

\begin{proof} Recall first the conclusion of
Example~B6.
In this example, we have a family of Ramsey domain fulfilling (P)
$({\mathcal F}_0^{\otimes l}, {\mathcal S}_{K}^{\otimes l})$
parametrized by natural numbers $K$ and $l$. We claim that
$S_{L,v_0}$, with $v_0\colon [L_0]\to [K_0]$ for some $L_0\leq L$,
is interpretable in $({\mathcal F}_0^{\otimes (L-L_0)}, {\mathcal S}_{K_0}^{\otimes (L-L_0)})$, which will prove the lemma
by Proposition~\ref{P:inter}.

Set
\[
l = L-L_0.
\]
Take $(S_3)^l\in {\mathcal S}_{K_0}^{\otimes l}$ and define
\[
\alpha\colon S_{L,v_0}\to (S_3)^l
\]
as follows. For a natural number $0\leq k\leq K_0$, let
$\tilde{k}\in S_3$ be the function
\begin{equation}\notag
\widetilde{k}(x)=\begin{cases} K_0, &\text{if $x=1$;}\\
k, &\text{if $x=2$;}\\
\max(1, K_0- 1), &\text{if $x=3$.}
\end{cases}
\end{equation}
Let $v\in S_{L,v_0}$ be such that $v\colon [L']\to [K_0]$. Define
\[
\alpha(v) = (\widetilde{v(L_0+1)}, \widetilde{v(L_0+2)}, \dots,
\widetilde{v(L')}, \widetilde{0}, \dots, \widetilde{0}),
\]
where we put $L-L'$ entries $\widetilde{0}$ at the end of the above
formula. It is clear that if $\partial_cv_1=\partial_cv_2$ for $v_1,v_2\in S_{L,v_0}$,
then $\partial_\pi (\alpha(v_1)) = \partial_\pi(\alpha(v_2))$. So condition (i) from the
definition of interpretation holds.

We now check condition (ii). Note that $(\prod_{i\leq l} F_{N_i, M_i})\dbullet
(S_3)^l$ is defined precisely when $M_i=3$ for all $i\leq l$. Fix
therefore $\prod_{i\leq l} F_{N_i, 3}\in  {\mathcal
F}_0^{\otimes l}$. For $N=\sum_{i\leq l}N_i$, $F_{L_0+N, L, L_0}\dbullet S_{L,
v_0}$ is defined in $({\mathcal F}_1, {\mathcal S}_1)$. We need to
describe a function
\[
\phi\colon \prod_{i\leq l} F_{N_i, 3} \to F_{L_0+N, L, L_0}.
\]
Let
\[
{\bar p}= (p_1, \dots , p_l)\in \prod_{i\leq l} F_{N_i,3}.
\]
Since $v_0$ is a surjection, we can fix $l_0, l_1\in [L_0]$ so that
\[
v_0(l_0) =K_0\;\hbox{ and }\; v_0(l_1)=\max(1, K_0- 1).
\]
For $x\in [L_0]$, let
\[
\phi({\bar p})(x) = x,
\]
for $x\in [L_0+N_1+\cdots +N_i]\setminus [L_0+N_1+\cdots +N_{i-1}]$,
let
\begin{equation}\notag
\phi({\bar p})(x) =
\begin{cases}
l_0, &\text{if $p_i(x-(L_0+N_1+\cdots+N_{i-1}))=1$;}\\
L_0+i, &\text{if $p_i(x-(L_0+N_1+\cdots+N_{i-1}))=2$;}\\
l_1, &\text{if $p_i(x-(L_0+N_1+\cdots+N_{i-1}))=3$.}
\end{cases}
\end{equation}

We check that \eqref{E:inter} of the definition of interpretability
holds. Note that, for $v\in S_{L,v_0}$, with $v\colon [L']\to
[K_0]$, the sequence
\[
\phi({\bar p}) \dowd v = v\circ \phi({\bar p})
\]
is the concatenation of the sequences
\[
v_0,\; \widetilde{v(L_0+1)}\circ p_1,\; \widetilde{v(L_0+2)}\circ
p_2, \dots,\; \widetilde{v(L')}\circ p_{L'-L_0},\; (K_0, \dots,
K_0),
\]
where the sequence of $K_0$-s at the end has length equal to the
size of $p_{L'-L_0+1}^{-1}(1)$; in particular, it has length $0$ if
$L'=L$. On the other hand,
\begin{equation}\notag
\begin{split}
&{\bar p}\dowd \alpha(v) =\\
&(\widetilde{v(L_0+1)}\circ p_1, \widetilde{v(L_0+2)}\circ p_2,
\dots , \widetilde{v(L')}\circ p_{L'-L_0}, \widetilde{0}\circ
p_{L'-L_0+1}, \dots, \widetilde{0}\circ p_l).
\end{split}
\end{equation}
Thus, \eqref{E:inter} follows since the size of
$p_{L'-L_0+1}^{-1}(1)$ is the number of entries equal to $K_0$ at
the beginning of $\widetilde{0}\circ p_{L'-L_0+1}$, and therefore
the above formula determines $\phi({\bar p}) \dowd v$.
\end{proof}

Theorem~\ref{T:main2} applied to the Ramsey domain
$({\mathcal F}_1, {\mathcal S}_1)$, which has (P) by Lemma~\ref{L:preHJ}, gives directly the Hales--Jewett
theorem as stated at the beginning of this subsection; we apply
Theorem~\ref{T:main2} with $t=K_0-1$, to $S_{K,v_0}\in {\mathcal
S}_1$ with $v_0\colon [L_0]\to [K_0]$.

\subsection{The Graham--Rothschild theorem}\label{Su:GR}

We outline here a proof of the Graham--Rothschild theorem, both the
original version \cite{GrRo71} and the partial rigid surjection
version isolated by Voigt \cite[Theorem 2.9]{Vo80}. Here are the two
statements that were already recalled in Subsection~\ref{Su:refor}.

\medskip

\noindent {\bf Graham--Rothschild.} {\em Given $d>0$, $K_0\leq K$,
$L_0\leq L$ and a rigid surjection $s_0\colon [L_0]\to [K_0]$, there
exists $M\geq L_0$ with the following property. For each
$d$-coloring of
\[
\{ s\colon [M]\to [K]\mid s\in {\rm RS}\hbox{ and }
s\upharpoonright [L_0] = s_0\}
\]
there is a rigid surjection $t_0\colon [M]\to [L]$, with
$t_0\upharpoonright [L_0] ={\rm id}_{[L_0]}$, such that
\[
\{ s\circ t_0\mid s\colon [L]\to [K],\, s\in {\rm RS}\hbox{ and }
s\upharpoonright [L_0] = s_0\}
\]
is monochromatic.}

\medskip
\noindent {\bf Graham--Rothschild, Voigt's version.} {\em Given
$d>0$, $K_0\leq K$, $L_0\leq L$ and a rigid surjection $s_0\colon
[L_0]\to [K_0]$, there exists $M\geq L_0$ with the following
property. For each $d$-coloring $c$ of
\[
\{ s\colon [M']\to [K]\mid L_0\leq M'\leq M\hbox{ and }
s\upharpoonright [L_0] = s_0\}
\]
there exist $M_0'$ and a rigid surjection $t_0\colon [M_0']\to [L]$,
with $L_0\leq M_0'\leq M$ and $t_0\upharpoonright [L_0] ={\rm
id}_{[L_0]}$, such that
\[
\{ s\circ t_0\mid s\colon [L']\to [K],\, L_0\leq L'\leq L ,\, s\in
{\rm RS}, \hbox{ and } s\upharpoonright [L_0] = s_0\}
\]
is monochromatic.}

\medskip

We use the Ramsey domain $({\mathcal F}_2, {\mathcal S}_2)$ over the composition space $(A_2, X_2)$
as defined in Subsection~\ref{Su:exact}. It is not difficult to
check from the Hales--Jewett theorem as stated in
Subsection~\ref{Su:HJ} that property (LP) holds for $({\mathcal F}_2, {\mathcal S}_2)$. So we have the following lemma.

\begin{lemma}
$({\mathcal F}_2, {\mathcal S}_2)$ fulfills (LP).
\end{lemma}

In the proof of the above lemma, the following obvious
observation plays a crucial role. If $s$ and $t$ are two non-empty rigid
surjections with
\[
\partial_f s = \partial_f t\colon [L]\to [K],
\]
then not only $s\res [L]= t\res [L]$, but also
\[
s\upharpoonright [L+1] = t\upharpoonright [L+1]
\]
as $s(L+1) = t(L+1) = K+1$.

We note the obvious facts that $({\mathcal F}_2, {\mathcal S}_2)$
is vanishing and linear. Now an application of Corollary~\ref{C:mainco2} to
$({\mathcal F}_2, {\mathcal S}_2)$ (using Lemmas ~\ref{L:22} and
\ref{L:g22}) yields the two versions of the Graham--Rothschild
theorem as stated at the beginning of this subsection.

\subsection{The self-dual Ramsey theorem}\label{Su:selpr}

We prove Theorem~\ref{T:RGR}. Recall the definition of connections
and their multiplication from Subsection~\ref{Su:self}. First we
state a reformulation of Theorem~\ref{T:RGR} and the usual partial
function version of this reformulation. We follow these
reformulations with an explanation of how the first one implies
Theorem~\ref{T:RGR}. Then, we give arguments for the two statements.

\medskip
\noindent {\bf Self-dual Ramsey theorem.} {\em Given $d>0$, $0<K\leq
L$, there exists $M$ with the following property. For each
$d$-coloring of
\[
\{ (s,p)\in {\rm AS}\mid (s,p)\colon [M]\to [K] \hbox{ and }
s^{-1}(K) = \{ M\}\}
\]
there is an augmented surjection $(t_0, q_0)\colon [M]\to [L]$ such
that $t_0^{-1}(L)=\{ M\}$ and
\[
\{ (t_0, q_0)\cdot (s,p)\mid (s,p)\colon [L]\to [K],\, (s,p)\in
{\rm AS}\hbox{ and } s^{-1}(K) = \{ L\}\}
\]
is monochromatic.}

\medskip

\noindent {\bf Self-dual Ramsey theorem; partial augmented
surjection version.} {\em Given $d>0$, $K\leq L$, there exists $M$
with the following property. For each $d$-coloring of
\[
\{ (s,p)\in {\rm AS}\mid (s,p)\colon [M']\to [K]\hbox{ for some }
M'\leq M \}
\]
there is an augmented surjection $(t_0, q_0)\colon [M_0']\to [L]$
for some $M_0'\leq M$ such that
\[
\{ (t_0, q_0)\cdot (s,p)\mid (s,p)\in {\rm AS},\, (s,p)\colon
[L']\to [K]\hbox{ for some }L'\leq L\}
\]
is monochromatic.}

\medskip

To obtain Theorem~\ref{T:RGR} from the first of the above
statements, associate with an increasing surjection $p\colon [L]\to
[K]$ an increasing injection $i_p\colon [K]\to [L]$ given by
$i_p(x)=\max p^{-1}(x)$. If $(s,p)$ is an augmented surjection with
$s,p\colon [L]\to [K]$, for some $0<K\leq L$, and $s^{-1}(K)=\{
L\}$, then
\[
(s\upharpoonright [L-1], i_p\upharpoonright [K-1])
\]
is a connection between $[L-1]$ and $[K-1]$ and each connection
between $[L-1]$ and $[K-1]$ is uniquely representable in this way.
Moreover, if $(t,q)$ is another augmented surjection with $t,q\colon
[M] \to [L]$ and with $t^{-1}(L) = \{ M\}$, then
\begin{equation}\notag
\begin{split}
((s\circ t)\upharpoonright [M-1]&, i_{p\circ q}\upharpoonright
[K-1])\\
&= (t\upharpoonright [M-1] , i_q\upharpoonright [L-1])\cdot
(s\upharpoonright [L-1], i_p\upharpoonright [K-1]),
\end{split}
\end{equation}
where the multiplication on the right hand side is the
multiplication of connections. These observations show that the first
of the above statements implies Theorem~\ref{T:RGR}.

The following lemma will turn out to be an immediate consequence of
the Hales--Jewett theorem and the Voigt version of the
Graham--Rothschild theorem.

\begin{lemma}\label{L:hjas}
$({\mathcal F}_3, {\mathcal S}_3)$ fulfills (LP).
\end{lemma}

\begin{proof} A moment of thought (using the way $\partial$ acts on subsets of
$X_2$) convinces us that to see (${\rm LP}$) it suffices to show
Conditions 1 and 2 below, for $L\geq K>0$ and $d>0$. To state these
conditions, fix $(s_0,p_0)\in {\rm AS}$, $s_0, p_0\colon [L_0]\to
[K-1]$ for some $L_0<L$. The role of the elements $x$ and $a$ in
(${\rm LP}$) is played by $(s_0, p_0)$ and $({\rm id}_{[L_0]}, {\rm
id}_{[L_0]})$, respectively. Note that $(t,q)$ from $F_{M,L}$ or $G_{M,L}$ extends
$({\rm id}_{[L_0]}, {\rm id}_{[L_0]})$ precisely when
$t\upharpoonright [L_0] = {\rm id}_{[L_0]}$ and
$q\upharpoonright [L_0+1] = {\rm id}_{[L_0+1]}$.

{\bf Condition 1.} {\em There exists $M\geq L$ such that for each
$d$-coloring of
\[
F_{M,L}.\{ (s,p)\in F_{L,K}\mid \partial (s,p) = (s_0,p_0)\}
\]
there exists $(t_0, q_0)\in F_{M,L}$ such that
\[
t_0\upharpoonright [L_0] = {\rm id}_{[L_0]}\;\hbox{ and }\;
q_0\upharpoonright [L_0+1] = {\rm id}_{[L_0+1]}
\]
and
\[
\{ (t_0,q_0)\dowd (s,p)\mid (s,p)\in F_{L,K},\, \partial (s,p) =
(s_0,p_0)\}
\]
is monochromatic.}

This statement amounts to proving the following result.

\noindent {\em There exists $M> L_0$ such that for each $d$-coloring
of all rigid surjections $t\colon [M-1]\to [K-1]$ with
$t\upharpoonright [L_0] = s_0$ there exists a rigid surjection
$t_0\colon [M-1]\to [L-1]$ such that $t_0\upharpoonright [L_0] =
{\rm id}_{[L_0]}$ and
\[
\{ s\circ t_0\mid s\colon [L-1]\to [K-1]\hbox{ a rigid surjection
and }s\upharpoonright [L_0] = s_0\}
\]
is monochromatic.}

This is a special case of the Hales--Jewett theorem, as stated and
proved in Subsection~\ref{Su:HJ}.

{\bf Condition 2.} {\em There exists $M\geq L$ such that for each
$d$-coloring of
\[
G_{M,L}.\{ (s,p)\in G_{L,K}\mid \partial (s,p) = (s_0,p_0)\}
\]
there exists $(t_0, q_0)\in G_{M,L}$ such that
\[
t_0\upharpoonright [L_0] = {\rm id}_{[L_0]}\;\hbox{ and }\;
q_0\upharpoonright [L_0+1] = {\rm id}_{[L_0+1]}
\]
and
\[
\{ (t_0,q_0)\dowd (s,p) \mid (s,p)\in G_{L,K},\, \partial (s,p) =
(s_0,p_0)\}
\]
is monochromatic.}

We will produce a sequence of statements the last of which will be Condition 2.

Fix $L_1\geq K$. We will specify later how large $L_1$ should be.
Statement (A) below follows from the Graham--Rothschild
theorem as stated in Subsection~\ref{Su:GR} in the same way as Condition 1 above
follows the Hales--Jewett theorem; we obtain $M-1$ from the Graham--Rothschild
theorem applied to $d$, $K$, $L_1$ and $s_0$.

\noindent (A) {\em There exists $M>L_0$ such that for each $d$-coloring
of all augmented surjections from $[M]$ to $[K]$
there exists an augmented  surjection $(t_1,q_1) \colon [M]\to [L_1]$
such that $t_1\upharpoonright [L_0] = {\rm id}_{[L_0]}$, $q_1\upharpoonright [L_0+1] = {\rm id}_{[L_0+1]}$, and
\[
\{ (t_1,q_1)\dowd (s,p) \mid (s,p)\colon [L_1]\to [K] \hbox{ and } \partial(s,p)= (s_0,p_0)\}
\]
is monochromatic. }

Using statement (A), we show below how to obtain the following statement.

\noindent (B) {\em There exists $M>L_0$ with the following property. For each $d$-coloring
of all augmented surjections from $[M']$ to $[K]$, where $M'$ is such that $L_0<M'\leq M$,
there exists an augmented
surjection $(t_1,q_1) \colon [M]\to [L_1]$
such that $t_1\upharpoonright [L_0] = {\rm id}_{[L_0]}$, $q_1\upharpoonright [L_0+1] = {\rm id}_{[L_0+1]}$, and
for each $L'$ with $L_0< L'\leq L_1$ the set
\[
\{ (t_1,q_1)\dowd (s,p) \mid (s,p)\colon [L']\to [K] \hbox{ and } \partial(s,p)= (s_0,p_0)\}
\]
is monochromatic. }

Statement (B) is clear for $L_1=L_0+1$. To move from $L_1$ to $L_1+1$ apply the inductive assumption to get $M_1$ that works
for $L_1$. Now apply statement (A) to $M_1+1$ (playing the role of $L_1$) obtaining $M$. One checks that this $M$ works for
statement (B). Indeed, assume we have a $d$-coloring $c$ of all augmented surjections from $[M']$ to $[K]$, where $M'$ is such that $L_0<M'\leq M$.
By (A), we find $(t_1,q_1)\colon [M]\to [M_1+1]$ such that
$t_1\upharpoonright [L_0] = {\rm id}_{[L_0]}$, $q_1\upharpoonright [L_0+1] = {\rm id}_{[L_0+1]}$, and
\[
\{ (t_1,q_1)\dowd (s,p) \mid (s,p)\colon [M_1+1]\to [K] \hbox{ and } \partial(s,p)= (s_0,p_0)\}
\]
is monochromatic. We color all the augmented surjections $(t,q)\colon [M']\to [K]$, where $L_0<M'\leq M_1$ by
letting
\[
c'((t,q)) = c((t_1,q_1)\dowd (t,q)).
\]
By the choice of $M_1$, there exists $(t_1',q_1')\colon [M_1]\to [L_1]$ such that the conclusion of (B) holds for it.
Let $(t_1'', q_1'')\colon [M_1+1]\to [L_1+1]$ be given by
\[
t_1'' \upharpoonright [M_1]= t_1',\, q_1''\upharpoonright [M_1]= q_1',\hbox{ and }t_1''(M_1+1) = q_1''(M_1+1) = L_1+1.
\]
Then it is easy to see that the augmented surjection $(t_1,q_1)\dowd (t_1'', q_1'')$ witnesses the conclusion of (B)
for coloring $c$.

\noindent (C) {\em There exists $M>L_0$ with the following property. For each $d$-coloring
of all augmented surjections from $[M']$ to $[K]$, where $L_0<M'\leq M$,
there exists an augmented
surjection $(t_0,q_0) \colon [M]\to [L+1]$
such that $t_0\upharpoonright [L_0] = {\rm id}_{[L_0]}$, $q_0\upharpoonright [L_0+1] = {\rm id}_{[L_0+1]}$, and
the set
\[
\{ (t_0,q_0)\dowd (s,p) \mid (s,p)\colon [L']\to [K],\, L_0< L'\leq L, \hbox{ and } \partial(s,p)= (s_0,p_0)\}
\]
is monochromatic. }

To see (C), first pick $L_1$ so that for each $d$-coloring of $[L_1-L_0]$, there is a subset of $[L_1-L_0]$
with $L+1-L_0$ elements that is monochromatic. Apply (B) to $L_1$ obtaining $M$. This $M$ works for (C).
Indeed, for a $d$-coloring as in (C) (and in (B)), we get $(t_1, q_1)\colon [M]\to [L_1]$ as in
the conclusion of (B). Now we use the choice of $L_1$ to find $(t_2, q_2)\colon [L_1]\to [L+1]$ with
$t_2=q_2$ and $q_2\upharpoonright [L_0+1] = {\rm id_{[L_0+1]}}$, and so that $(t_0, q_0)$ defined to be $(t_1,q_1)\dowd (t_2,q_2)$
is as required by (C).

Now, (C) is easily seen to be just a reformulation of Condition 2. Thus, (${\rm LP}$) holds and the lemma follows.
\end{proof}

Since $({\mathcal F}_3, {\mathcal S}_3)$ is clearly vanishing,
by Lemmas~\ref{L:as}  and \ref{L:SD}, Corollary~\ref{C:mainco} can
be applied to $({\mathcal F}_3, {\mathcal S}_3)$ yields the statements from the beginning
of this subsection.

\section{Walks, a limiting case}\label{S:wal}

In this section, we give a natural limiting example of the extent
of condition (R). The motivation for
this example comes from \cite{IrSo05} and is related to a problem of
Uspenskij \cite{Us00}.

Recall that walks were defined in Subsection~\ref{Su:injsur}, and the set of all walks was denoted there by $W$.
Let $C=Z=W$. We note that $C\subseteq A_2$ and $Z\subseteq X_2$, as
defined in Subsection~\ref{Su:exba}. We equip $C$ with the
multiplication inherited from $A_2$ and we take the action
of $C$ on $Z$ to be the one inherited from $(A_2, X_2)$. Note also
that $Z$ is closed under the forgetful truncation $\partial_f$ with
which $(A_2, X_2)$ is equipped. We take it as the truncation on $(C,
Z)$. We also consider the function $|\cdot |$ defined on $X_2$, we
restrict it to $Z$, and denote it again by $|\cdot |$. The following
lemma is an immediate consequence of Lemma~\ref{L:22}.

\begin{lemma}\label{L:44}
$(C, Z)$ with the operations defined above is a normed composition space.
\end{lemma}

Let ${\mathcal H} = {\mathcal W}$ consist of all finite non-empty subsets $F$ of $\rm W$
such that for some $K>0$ each element of $F$ is a walk with range $[K]$. We write
$r(F)=K$. We write $d(F)=L$ if $L$ is the largest natural number such that there is a walk
in $F$ whose domain is $[L]$. We call $F$ {\em tidy} if each walk in $F$ has domain $[d(F)]$.
For $F_{1}, F_{2}\in {\mathcal H} = {\mathcal W}$,
$F_{1}\bullet F_{2}$ and $F_{1}\dbullet F_{2}$ are defined
if and only if $d(F_2)=r(F_1)$ and $F_2$ being tidy implies that $F_1$ is tidy.
Then we let
\[
F_{1}\bullet F_{2} = F_{1}\dbullet F_{2} = F_{1}\cdot F_{2} = F_{1}\dowd F_{2}.
\]
One easily checks the following lemma.

\begin{lemma}\label{L:455}
$({\mathcal H}, {\mathcal W})$ with $\bullet$
and $\dbullet$ is a Ramsey domain over $(C,Z)$.
\end{lemma}

The Ramsey statement equivalent to
this Ramsey domain fulfilling condition (R)
was motivated by a question of Uspenskij \cite{Us00}, which asked if the
universal minimal flow of the homeomorphism group of the pseudo-arc
is the pseudo-arc itself together with the natural action of the
homeomorphism group.
Uspenskij's question would have had a positive answer if this Ramsey statement were true. (The Ramsey statement
would imply such an answer by \cite{IrSo05} and by a dualization of the techniques of \cite{KePeTo05}.)
However, the theorem below implies that the Ramsey statement is false.

\begin{theorem}\label{T:color}
For every $M\geq 3$ there exists a $2$-coloring of all walks from
$[M]$ to $[3]$ such that for each walk $t\colon [M]\to [6]$ the set
\[
\{ s\circ t\mid s\colon [6]\to [3]\hbox{ a walk}\}
\]
is not monochromatic.
\end{theorem}

\begin{proof} We show a bit more: to contradict monochromaticity we
only need a set of walks $s\colon [6]\to [3]$ that differ at two
elements of their common domain. Let
\[
S = \{ s\colon [6]\to [3]\mid s(1)=1, s(2)=s(5)=2, s(6)=3, s(3),
s(4)\in \{1,2\}\}.
\]
Clearly each element of $S$ is a walk from $[6]$ to $[3]$. We claim
that for each $M\geq 3$ there is a $2$-coloring of all walks from
$[M]$ to $[3]$ such that for each walk $t\colon [M]\to [6]$ the set
$S\circ t$ is not monochromatic.

Let $M\geq 3$. For a walk $u\colon [M]\to [3]$ define
\[
a(u) = |\{ y\in [M]\mid u(x)\leq 2 \hbox{ for all }x\leq y,\;
u(y)=1, \hbox{ and }u(y+1) = 2\}|.
\]
Define a $2$-coloring $c$ by letting
\[
c(u)=a(u)\mod 2.
\]
We claim that this coloring is as required.

Let $t\colon [M]\to [6]$ be a walk. We analyze $t$ in order to
compute $a(s\circ t)$ for $s\in S$ in terms of certain numbers
associated with $t$. Let $M_0\leq M$ be the smallest natural number
with $t(M_0)=6$. There exist unique, pairwise disjoint intervals
$I\subseteq [M_0]$ that are maximal with respect to the property
$t(I)\subseteq \{3,4\}$. For such an $I$, let $I^-$ and $I^+$ be
$(\min I)-1$ and $(\max I)+1$, respectively. Note that $t(I^-),
t(I^+)\in \{ 2,5\}$. We distinguish four types of such intervals
$I$:
\begin{equation}\notag
\begin{split}
I\in P_1&\Longleftrightarrow t(I^-)=2\hbox{ and }t(I^+)=5;\\
I\in P_2&\Longleftrightarrow t(I^-)=5\hbox{ and }t(I^+)=2;\\
I\in Q_1&\Longleftrightarrow t(I^-)=t(I^+)=2;\\
I\in Q_2&\Longleftrightarrow t(I^-)=t(I^+)=5.
\end{split}
\end{equation}
Note right away that since $t$ is a walk, $t(1)=1$, $t(M_0-1)=5$,
and $t(x)\in \{ 1, \dots, 5\}$ for $x\in [M_0-1]$, it follows that
$|P_1|-|P_2|=1$ and therefore
\begin{equation}\label{E:sump}
|P_1|+|P_2| \hbox{ is odd.}
\end{equation}

For each $I\in P_1\cup P_2\cup Q_1\cup Q_2$, define $a_t(I)$ as
follows.
\begin{equation}\notag
a_t(I)=\begin{cases} |\{ x\in I\mid t(x)=3,\, t(x+1)=4\}|,&
\text{ if $I\in P_1\cup Q_1$;}\\
|\{ x\in I\mid t(x)=4,\, t(x+1)=3\}|,& \text{ if $I\in P_2\cup
Q_2$.}
\end{cases}
\end{equation}
Note that for $I$ in $Q_1$ or $Q_2$, the two cases in the above
definition give the same value for $a_t(I)$. Further, let
\[
a_t(*) = |\{ x\in [M_0]\mid t(x)=1,\, t(x+1)=2\}|.
\]

Recall now the set $S$ introduced in the beginning of the proof. We can write $S = \{ s_1, s_2, s_3, s_4\}$, where $s_1,\, s_2,\,
s_3$ and $s_4$ are determined by the conditions
\begin{equation}\notag
\begin{split}
&s_1(3)=s_1(4)=1,\\
&s_2(3)=2,\, s_2(4)=1,\\
&s_3(3)=1,\, s_3(4)=2, \hbox{ and}\\
&s_4(3)=s_4(4)=2.
\end{split}
\end{equation}
An inspection convinces us that
\begin{equation}\notag
\begin{split}
a&(s_i\circ t)\\
&= \begin{cases} a_t(*)+ \sum_{I\in P_1\cup P_2\cup
Q_1\cup Q_2} 1,&\text{if $i=1$;}\\
a_t(*)+\sum_{I\in P_1\cup P_2}a_t(I)+\sum_{I\in Q_1}a_t(I)+\sum_{I\in
Q_2} (a_t(I)+1), &\text{if $i=2$;}\\
a_t(*)+\sum_{I\in P_1\cup P_2}a_t(I)+\sum_{I\in Q_1}(a_t(I)+1)
+\sum_{I\in Q_2} a_t(I), &\text{if $i=3$;}\\
a_t(*), &\text{if $i=4$.}
\end{cases}
\end{split}
\end{equation}

Assume towards a contradiction that a walk $t\colon [M]\to [6]$ is
such that $S\circ t$ is monochromatic. It follows from the above
expressions for $a(s_i\circ t)$ that the numbers
\begin{equation}\notag
\begin{split}
&\sum_{I\in P_1\cup P_2\cup Q_1\cup Q_2} 1, \\
&\sum_{I\in P_1\cup P_2}a_t(I)+\sum_{I\in Q_1}a_t(I)+\sum_{I\in
Q_2} (a_t(I)+1),\\
&\sum_{I\in P_1\cup P_2}a_t(I)+\sum_{I\in Q_1}(a_t(I)+1)
+\sum_{I\in Q_2} a_t(I),\\
&\hbox{ and }0
\end{split}
\end{equation}
have the same parity, that is, since the last number is $0$, they must
all be even. Now it follows from the first line that
\begin{equation}\label{E:sume}
|P_1|+|P_2|+|Q_1|+|Q_2|\hbox{ is even}
\end{equation}
and, by subtracting the second line from the third one, that
$|Q_1|-|Q_2|$ is even, and so
\begin{equation}\label{E:suml}
|Q_1|+|Q_2|\hbox{ is even.}
\end{equation}
Equations \eqref{E:sume} and \eqref{E:suml} imply that the natural
number $|P_1|+|P_2|$ is even contradicting \eqref{E:sump}.
\end{proof}

\section{A problem}\label{S:pro}
Observe that the failure of a Ramsey result from
Section~\ref{S:wal}, the classical Ramsey theorem, and the
Graham--Rothschild theorem can be viewed in a uniform way as
statements about composition spaces that are closely related to each other.
Fix $A$ so that ${\rm IS}\subseteq A\subseteq {\rm RS}$, where the
families of increasing surjections ${\rm IS}$ and of rigid
surjections ${\rm RS}$ are defined in Subsection~\ref{Su:injsur}.
Assume that for $s, t\in A$ with $s\colon [L]\to [K]$ and $t\colon
[N]\to [M]$ with $L\leq M$, the canonical composition $s\circ t$, as
defined in Subsection~\ref{Su:canco}, is in $A$. Assume further that
$\partial_f s\in A$ for $s\in A$, where the forgetful truncation
$\partial_f$ is defined in Subsection~\ref{Su:extrnc}. Consider
the composition space $(A,A)$ in which multiplication on $A$ is the same as
the action of $A$ on $A$, both are defined precisely when $s\circ t$
is defined and are given by
\[
t\cdot s = t\dowd s = s\circ t,
\]
and in which the truncation operator is given by $\partial_f$.
The following problem presents itself.

\noindent {\em Find all sets $A$ as above such that the following Ramsey statement holds:
given $d>0$ and $K,L$ there is $M$ such that for each $d$ coloring of
$\{ s\in A\mid s\colon [M]\to [K]\}$ there exists $t_0\in A$ with $t_0\colon[M]\to [L]$ and such that
\[
\{ s\circ t_0\mid s\in A,\, s\colon [L]\to [K]\}.
\]
is monochromatic.}

Note that if $A={\rm RS}$ or $A= {\rm IS}$, then the answer is positive.
In the first case, the resulting Ramsey theorem is the dual Ramsey
theorem (see Subsection~\ref{Su:GR}) and in the second case it is a
reformulation of the classical Ramsey theorem (see Example A5 in
Subsection~\ref{S:aux}). Note also that, by Theorem~\ref{T:color},
if $A$ is the set of all walks ${\rm W}$, then the answer is negative.

\medskip

\noindent {\bf Acknowledgement.} I would like to thank the referee for a very thorough reading of the
paper and for suggesting a number of important changes to the presentation of the material in it.
I also thank Alekos Kechris and Min Zhao for their remarks on an earlier version of the paper.

\end{document}